\documentclass[oneside,reqno,english]{amsart}
\usepackage[T1]{fontenc}
\usepackage[utf8]{inputenc}
\usepackage{xcolor}
\usepackage{babel}
\usepackage{prettyref}
\usepackage{amstext}
\usepackage{amsthm}
\usepackage{amssymb}
\usepackage[pdfusetitle,
 bookmarks=true,bookmarksnumbered=false,bookmarksopen=false,
 breaklinks=false,pdfborder={0 0 0},pdfborderstyle={},backref=false,colorlinks=false]
 {hyperref}
\hypersetup{
 colorlinks=true,citecolor=blue,linkcolor=blue,linktocpage=true}

\makeatletter
\numberwithin{equation}{section}
\numberwithin{figure}{section}

\usepackage{prettyref}

\newrefformat{cor}{Corollary~\ref{#1}}
\newrefformat{subsec}{Section~\ref{#1}}
\newrefformat{lem}{Lemma~\ref{#1}}
\newrefformat{thm}{Theorem~\ref{#1}}
\newrefformat{sec}{Section~\ref{#1}}
\newrefformat{chap}{Chapter~\ref{#1}}
\newrefformat{prop}{Proposition~\ref{#1}}
\newrefformat{exa}{Example~\ref{#1}}
\newrefformat{tab}{Table~\ref{#1}}
\newrefformat{rem}{Remark~\ref{#1}}
\newrefformat{def}{Definition~\ref{#1}}
\newrefformat{fig}{Figure~\ref{#1}}
\newrefformat{claim}{Claim~\ref{#1}}
\newrefformat{assu}{Assumption~\ref{#1}}

\makeatother

\providecommand\theoremname{Theorem}
\theoremstyle{plain}
\newtheorem{thm}{\protect\theoremname}[section]
\providecommand\definitionname{Definition}
\theoremstyle{definition}
\newtheorem{defn}[thm]{\protect\definitionname}
\providecommand\lemmaname{Lemma}
\theoremstyle{plain}
\newtheorem{lem}[thm]{\protect\lemmaname}
\providecommand\corollaryname{Corollary}
\newtheorem{cor}[thm]{\protect\corollaryname}
\providecommand\remarkname{Remark}
\theoremstyle{remark}
\newtheorem{rem}[thm]{\protect\remarkname}
\providecommand\propositionname{Proposition}
\theoremstyle{plain}
\newtheorem{prop}[thm]{\protect\propositionname}
\begin{document}
\title[Tree Capacity and Splitting Isometries for Subinvariant Kernels]{Tree Capacity and Splitting Isometries for Subinvariant Kernels}
\begin{abstract}
Starting from a subinvariant positive definite kernel under a branching
pullback, we attach to the resulting kernel sequence an electrical
network on the word tree whose edge weights are the diagonal increments.
The resulting effective resistance and capacity give explicit criteria
and quantitative bounds for the diagonal growth, together with a matching
upper bound under a mild level regularity condition. When the diagonal
sequence has finite limit at a point, we prove convergence of the
full kernels and obtain an invariant completion with a minimality
property. We also describe the associated RKHS splitting and a boundary
martingale construction, and use the boundary limits to obtain weighted
positive definite kernels and invariant majorants.
\end{abstract}

\author{James Tian}
\address{Mathematical Reviews, 535 W. William St, Suite 210, Ann Arbor, MI
48103, USA}
\email{james.ftian@gmail.com}
\maketitle

\section{Introduction}\label{sec:1}

Let $X$ be a set and let $\varphi_{1},\dots,\varphi_{m}:X\to X$
be given maps. A positive definite kernel $J$ on $X$ may be pulled
back along the family $\varphi=\left(\varphi_{i}\right)$ by 
\[
\left(LJ\right)\left(s,t\right):=\sum^{m}_{i=1}J\left(\varphi_{i}\left(s\right),\varphi_{i}\left(t\right)\right).
\]
Throughout we fix a positive definite kernel $K$ and assume the subinvariance
inequality 
\[
LK\ge K
\]
in the Loewner order. 

Starting from $K_{0}:=K$ we form the increasing sequence of kernels
\[
K_{n+1}:=LK_{n}
\]
and we write the diagonals 
\[
u_{n}\left(s\right):=K_{n}\left(s,s\right)
\]
for all $n\geq0$.

We show that the scalar increments of $\left(u_{n}\right)$ determine
a weighted rooted $m$-ary tree, which may be viewed as an electrical
network (\prettyref{sec:2}). Fix a basepoint $s\in X$ and a word
$w$ of length $\left|w\right|$. At the corresponding depth, consider
the one-step diagonal increment 
\[
a_{s}\left(w\right):=u_{\left|w\right|+1}\left(\varphi_{w}\left(s\right)\right)-u_{\left|w\right|}\left(\varphi_{w}\left(s\right)\right)\ge0.
\]
We use these numbers to assign conductances to edges of the word tree.
The network has no ambient geometry or metric assumptions. It is determined
by positivity and the recursion for the kernels. Classical energy
and capacity estimates on trees can then be expressed in terms of
the diagonal values $u_{n}(s)$.

The conductances of this network are determined directly by the diagonal
sequence. Indeed, 
\[
\sum_{w\in W_{k}}a_{s}(w)=u_{2k+1}(s)-u_{2k}(s).
\]
This identity relates the effective resistance of the tree directly
to the values $u_{n}(s)$. In particular, positive limiting capacity
implies (\prettyref{thm:lower-bound}) 
\[
\sum^{\infty}_{k=0}\frac{m^{k}}{u_{2k+1}(s)-u_{2k}(s)}<\infty.
\]
Under a mild regularity condition on the increments within each level,
the converse also holds (\prettyref{cor:capacity-criterion}).

The bounded case leads in a different direction. Suppose that $u_{n}(s)$
has a finite limit. Positivity then controls the off-diagonal increments
as well, and the kernels $K_{n}$ converge pointwise on the set 
\[
X_{\mathrm{fin}}=\left\{ s\in X:\lim_{n\to\infty}u_{n}(s)<\infty\right\} .
\]
Their limit $K_{\infty}$ is positive definite and satisfies (\prettyref{thm:3-4})
\[
LK_{\infty}=K_{\infty}.
\]
Moreover, $K_{\infty}$ is the smallest $L$-invariant positive definite
kernel that dominates $K$ on $X_{\mathrm{fin}}$ (\prettyref{thm:3-4}).

The identity $LK_{\infty}=K_{\infty}$ also gives an isometric decomposition
of the reproducing kernel Hilbert space $\mathcal{H}(K_{\infty})$;
see \prettyref{prop:4-6}. Following a random branch of the word tree
gives a corresponding martingale representation of $K_{\infty}$ (\prettyref{lem:5-1}).
Under an $L^{2}$ condition, these martingales have boundary limits
(\prettyref{thm:5-5}). The boundary limits may be integrated against
nonnegative weights to produce further positive definite kernels,
and the action of $L$ on these kernels corresponds to the shift on
the boundary (Theorems \ref{thm:7-2}--\ref{thm:7-3}).

\subsection*{Related literature}

Positive definite kernels and their reproducing kernel Hilbert spaces
provide the basic setting for the paper \cite{MR51437,MR2284176}.
Questions about domination and extension of positive definite kernels
have been studied in several forms \cite{MR3251728,MR3560890}. Function
spaces and Carleson measures on homogeneous trees are studied in \cite{MR3490548}.
Harmonic functions and martingales on trees appear in \cite{MR3545231,MR4232185}.

Our earlier paper \cite{tian2026sub} studied the same subinvariant
iteration in view of invariant majorants. It gave a minimal invariant
completion under diagonal boundedness, together with a defect-space
realization, the associated one-step RKHS isometry, and boundary representations
based on the minimal diagonal harmonic majorant and its Doob transforms.
The present paper takes a different direction from the diagonal increments
themselves. They determine conductances on the word tree, and hence
an electrical network whose effective resistance and capacity give
quantitative information about the diagonal growth. The RKHS isometry
is developed along the full word tree through the isometries $V_{n}$,
and the boundary construction uses a single product measure and the
martingales obtained by sampling $K_{\infty}$ along random branches.
This leads to the shift identities and weighted boundary kernels considered
below.

Effective resistance, energy, and capacity on trees and infinite networks
are classical \cite{MR920811,MR1324344,MR3616205,MR4641241}. The
cutset estimate used below (\prettyref{thm:lower-bound}) is the classical
Nash-Williams inequality. Dirichlet form theory gives a broader potential-theoretic
setting \cite{MR2778606}. Resistance methods on fractals, and related
questions for Dirichlet forms, appear in \cite{MR1840042,MR2964679,MR4068276,MR4730446}.
Here the conductances are not given in advance by a graph or an ambient
geometry. They come from the diagonal increments of the kernels.

The identity $LK_{\infty}=K_{\infty}$ leads to the Hilbert space
and boundary constructions in the later sections. Related isometric
splittings and Wold type decompositions appear in \cite{MR2416726,MR4651635}.
The boundary argument uses standard martingale ideas on the word space
and is related to boundary and Martin boundary constructions on trees
\cite{MR1814344,MR1743100,MR284569,MR4031266,MR4412972,MR4388381,MR4730766}. 

\section{Tree capacity and energy}\label{sec:2}

Fix a point $s\in X$. The maps $\varphi_{1},\dots,\varphi_{m}$ generate
a rooted tree from $s$ as follows. At the first level are the points
$\varphi_{i}\left(s\right)$, at the next level the points $\varphi_{j}\varphi_{i}\left(s\right)$,
and so on. We use this tree to keep track of the diagonal growth of
the kernels $K_{n}$.

The assumption $LK\ge K$ makes the diagonal sequence 
\[
u_{n}\left(x\right)=K_{n}\left(x,x\right)
\]
increasing. At a vertex $w$ of depth $k$, consider the increment
\[
u_{k+1}\left(\varphi_{w}\left(s\right)\right)-u_{k}\left(\varphi_{w}\left(s\right)\right).
\]
After a normalization by the depth, we use this as the conductance
of the edges leaving $w$. The kernel iteration then gives an electrical
network on the rooted tree. Sending one unit of flow from the root
through this network gives an effective resistance to each level.

The network approach is classical. But here the edge weights come
from the diagonal growth of a positive kernel, with no geometric assumptions
on $X$. The level sums of these increments have the simple form in
\prettyref{lem:telescoping}. This reduces the cutset estimates to
the scalar series in \prettyref{eq:2-21}. Under the level concentration
condition \prettyref{eq:2-18}, its finiteness is equivalent to positive
limiting capacity.

We begin with the kernel iteration and the elementary facts that will
be used below.

Fix a set $X$ and maps $\varphi_{1},\dots,\varphi_{m}:X\to X$. For
$n\ge0$ define the word sets 
\[
W_{n}:=\left\{ 1,\dots,m\right\} ^{n},\qquad W_{*}:=\bigcup_{n\ge0}W_{n},
\]
with $W_{0}=\left\{ \emptyset\right\} $. For any $w=i_{1}\cdots i_{n}\in W_{n}$
set 
\[
\varphi_{w}:=\varphi_{i_{n}}\circ\cdots\circ\varphi_{i_{1}}.
\]

\begin{defn}
Let $\mathcal{K}\left(X\right)$ be the set of all positive definite
(p.d.) kernels on $X$. Given $\varphi$, the associated pullback
operator on $\mathcal{K}\left(X\right)$ is given by 
\begin{equation}
\left(LJ\right)\left(s,t\right):=\sum^{m}_{i=1}J\left(\varphi_{i}\left(s\right),\varphi_{i}\left(t\right)\right).\label{eq:2-1}
\end{equation}
\end{defn}

\begin{lem}
\label{lem:2-2}$L$ is order-preserving, i.e., $J_{1}\leq J_{2}\Longrightarrow LJ_{1}\leq LJ_{2}$.
\end{lem}

\begin{proof}
For any $J\in\mathcal{K}\left(X\right)$, using \prettyref{eq:2-1},
\[
\sum_{\alpha,\beta}\overline{c_{\alpha}}c_{\beta}\left(LJ\right)\left(s_{\alpha},s_{\beta}\right)=\sum^{m}_{k=1}\sum_{\alpha,\beta}\overline{c_{\alpha}}c_{\beta}J\left(\varphi_{k}\left(s_{\alpha}\right),\varphi_{k}\left(s_{\beta}\right)\right)\geq0
\]
for all finite subsets $\left\{ \left(c_{\alpha},s_{\alpha}\right)\right\} \subset\mathbb{C}\times X$,
by positive definiteness of $J$. Thus $LJ$ is also in $\mathcal{K}\left(X\right)$.
Therefore, 
\[
J_{1}\leq J_{2}\Longleftrightarrow J_{2}-J_{1}\in\mathcal{K}\left(X\right)\Longrightarrow L\left(J_{2}-J_{1}\right)\in\mathcal{K}\left(X\right)\Longleftrightarrow LJ_{1}\leq LJ_{2}.
\]
\end{proof}

\begin{defn}
Fix $K\in\mathcal{K}\left(X\right)$, and assume 
\begin{equation}
LK\ge K.\label{eq:2-2}
\end{equation}
Define the sequence of p.d. kernels $\left(n\in\mathbb{N}_{0}:=\left\{ 0,1,2,\dots\right\} \right)$
\begin{align}
K_{0} & :=K,\quad K_{n+1}:=LK_{n}\label{eq:2-3}
\end{align}
and the diagonals 
\begin{equation}
u_{n}\left(s\right):=K_{n}\left(s,s\right)\label{eq:2-4}
\end{equation}
for all $s\in X$.
\end{defn}

\begin{cor}
The following hold:
\begin{enumerate}
\item $K_{n+1}\ge K_{n}$, $n\in\mathbb{N}_{0}$;
\item $u_{n+1}\left(s\right)\ge u_{n}\left(s\right)$, $n\in\mathbb{N}_{0}$,
$s\in X$. 
\end{enumerate}
\end{cor}

\begin{proof}
By \prettyref{lem:2-2}, the assumption \prettyref{eq:2-2} implies
that $K_{n+1}\ge K_{n}$ for all $n$.

The pointwise claim may be verified as follows: set $D_{n}=K_{n+1}-K_{n}$,
so $D_{n}\in\mathcal{K}\left(X\right)$ and $D_{n}\left(s,s\right)=u_{n+1}\left(s\right)-u_{n}\left(s\right)$
by \prettyref{eq:2-4}. Then 
\[
D_{n}\left(s,s\right)=\left\langle D_{n}\left(\cdot,s\right),D_{n}\left(\cdot,s\right)\right\rangle _{H_{D_{n}}}\geq0
\]
where $H_{D_{n}}$ is the reproducing kernel Hilbert space of $D_{n}$.
\end{proof}

The inequalities above show that the diagonal values $u_{n}(s)$ form
an increasing sequence. We now attach these increments to a rooted
$m$-ary tree generated by the maps $\varphi_{i}$. The construction
gives a canonical network whose edge weights encode how much new diagonal
mass is created at each step.
\begin{defn}
Fix a basepoint $s\in X$. Consider the rooted $m$-ary tree with
vertex set $W_{*}$, root $\emptyset$, and edges $\left(w,wi\right)$
for $w\in W_{*}$ and $i\in\left\{ 1,\dots,m\right\} $. 

For each node $w\in W_{*}$, let 
\begin{equation}
a_{s}\left(w\right):=u_{\left|w\right|+1}\left(\varphi_{w}\left(s\right)\right)-u_{\left|w\right|}\left(\varphi_{w}\left(s\right)\right)\ge0,\label{eq:2-5}
\end{equation}
and for each edge $e=\left(w,wi\right)$, define the conductance 
\begin{equation}
c_{s}\left(e\right):=\frac{a_{s}\left(w\right)}{m^{\left|w\right|+1}}.\label{eq:2-6}
\end{equation}
\end{defn}

\begin{rem}
Equation \prettyref{eq:2-5} is the one-step diagonal increment at
the point $\varphi_{w}\left(s\right)$, taken at the index $n=\left|w\right|$
matching the depth of $w$. It measures how much the diagonal value
at $\varphi_{w}\left(s\right)$ increases when one passes from $K_{\left|w\right|}$
to $K_{\left|w\right|+1}=LK_{\left|w\right|}$. Equivalently, it is
the local amount of new diagonal mass created by a single application
of $L$ at the node $w$. 

The quantity $c_{s}$ in \prettyref{eq:2-6} is an edge weight. It
is chosen so that the total weight of the level-$k$ cutset $E_{k}$
takes the simple form 
\begin{equation}
C_{k}\left(s\right):=\sum_{e\in E_{k}}c_{s}\left(e\right)=\frac{1}{m^{k}}\sum_{w\in W_{k}}a_{s}\left(w\right),\label{eq:2-7}
\end{equation}
since each vertex $w\in W_{k}$ contributes $m$ edges, each of weight
$a_{s}\left(w\right)/m^{k+1}$. In particular, by \prettyref{lem:telescoping},
\[
C_{k}\left(s\right)=\frac{u_{2k+1}\left(s\right)-u_{2k}\left(s\right)}{m^{k}}.
\]
This normalization is compatible with the energy functional introduced
below. In fact, with these edge weights, level cutsets give natural
lower bounds for the energy of a unit flow, and hence for the effective
resistance to depth $N$. (Recall that, the cutset $E_{k}$ consists
of edges from level $k$ to level $k+1$, namely edges $\left(w,wi\right)$
with $\left|w\right|=k$.)
\end{rem}

\begin{defn}
Fix a depth $N\ge1$. A unit flow to depth $N$ is a nonnegative function
$\theta$ on edges with $\left|w\right|\le N-1$ such that 
\begin{equation}
\sum^{m}_{i=1}\theta\left(\emptyset,i\right)=1,\label{eq:2-8}
\end{equation}
and for every internal vertex $w$ with $0<\left|w\right|<N$, 
\begin{equation}
\theta\left(w^{-},w\right)=\sum^{m}_{i=1}\theta\left(w,wi\right),\label{eq:2-9}
\end{equation}
where $w^{-}$ denotes the parent of $w$.

The flow energy is 
\begin{equation}
\mathcal{E}_{s}\left(\theta\right):=\sum_{e}\frac{\theta\left(e\right)^{2}}{c_{s}\left(e\right)},\label{eq:2-10}
\end{equation}
summing over all edges up to depth $N$. 

Define the effective resistance and conductance by 
\begin{align}
R_{N}\left(s\right) & :=\inf\left\{ \mathcal{E}_{s}\left(\theta\right):\theta\text{ is a unit flow to depth }N\right\} ,\\
\mathrm{Cap}_{N}\left(s\right) & :=\frac{1}{R_{N}\left(s\right)}.
\end{align}
\end{defn}

A key identity for the rest of the construction is \prettyref{eq:2.11}
below. Note the increments $a_{s}(w)$ are defined at individual nodes,
but their total mass across each level admits a closed-form identity
depending only on the diagonal sequence $u_{n}(s)$. This will lead
to all subsequent capacity estimates. 
\begin{lem}
\label{lem:telescoping}For every $k\ge0$, 
\begin{equation}
\sum_{w\in W_{k}}a_{s}\left(w\right)=u_{2k+1}\left(s\right)-u_{2k}\left(s\right).\label{eq:2.11}
\end{equation}
In particular, for every $N\ge1$, 
\begin{align}
\sum^{N-1}_{k=0}\sum_{w\in W_{k}}a_{s}\left(w\right) & =\sum^{N-1}_{k=0}\left(u_{2k+1}\left(s\right)-u_{2k}\left(s\right)\right)\nonumber \\
 & =\left(u_{1}-u_{0}\right)+\left(u_{3}-u_{2}\right)+\cdots+\left(u_{2N-1}-u_{2N-2}\right).\label{eq:2.12}
\end{align}
\end{lem}

\begin{proof}
Fix $k\ge0$. By definition \prettyref{eq:2-5}, 
\[
\sum_{w\in W_{k}}a_{s}\left(w\right)=\sum_{w\in W_{k}}\left(u_{k+1}\left(\varphi_{w}\left(s\right)\right)-u_{k}\left(\varphi_{w}\left(s\right)\right)\right).
\]
On the diagonal, iterating $L$ gives 
\begin{align*}
u_{n}\left(x\right) & =K_{n}\left(x,x\right)=\left(L^{n}K\right)\left(x,x\right)\\
 & =\sum_{v\in W_{n}}K\left(\varphi_{v}\left(x\right),\varphi_{v}\left(x\right)\right)=\sum_{v\in W_{n}}u_{0}\left(\varphi_{v}\left(x\right)\right).
\end{align*}
Therefore, 
\begin{align*}
\sum_{w\in W_{k}}u_{k+1}\left(\varphi_{w}\left(s\right)\right) & =\sum_{w\in W_{k}}\sum_{v\in W_{k+1}}u_{0}\left(\varphi_{v}\left(\varphi_{w}\left(s\right)\right)\right)\\
 & =\sum_{w\in W_{k}}\sum_{v\in W_{k+1}}u_{0}\left(\varphi_{wv}\left(s\right)\right).
\end{align*}
The concatenation map $\left(w,v\right)\mapsto wv$ is a bijection
$W_{k}\times W_{k+1}\to W_{2k+1}$, so the last expression equals
\[
\sum_{u\in W_{2k+1}}u_{0}\left(\varphi_{u}\left(s\right)\right)=u_{2k+1}\left(s\right).
\]
Similarly, 
\[
\sum_{w\in W_{k}}u_{k}\left(\varphi_{w}\left(s\right)\right)=\sum_{w\in W_{k}}\sum_{v\in W_{k}}u_{0}\left(\varphi_{wv}\left(s\right)\right)=\sum_{u\in W_{2k}}u_{0}\left(\varphi_{u}\left(s\right)\right)=u_{2k}\left(s\right).
\]
Subtracting these two gives \prettyref{eq:2.11}. Summing over $k=0,\dots,N-1$
gives the stated formula \prettyref{eq:2.12}. 
\end{proof}

We next express the conductance of a level cutset and obtain a universal
lower bound on the energy of any unit flow.
\begin{lem}
\label{lem:cutset}For every $k\ge0$, 
\[
C_{k}\left(s\right)=\frac{1}{m^{k}}\sum_{w\in W_{k}}a_{s}\left(w\right)=\frac{u_{2k+1}\left(s\right)-u_{2k}\left(s\right)}{m^{k}}.
\]
Moreover, for every $N\ge1$ and every unit flow $\theta$ to depth
$N$, 
\begin{equation}
\mathcal{E}_{s}\left(\theta\right)\ge\sum^{N-1}_{k=0}\frac{1}{C_{k}\left(s\right)}.\label{eq:2-15}
\end{equation}
Consequently, 
\[
R_{N}\left(s\right)\ge\sum^{N-1}_{k=0}\frac{1}{C_{k}\left(s\right)},\qquad\mathrm{Cap}_{N}\left(s\right)\le\left(\sum^{N-1}_{k=0}\frac{1}{C_{k}\left(s\right)}\right)^{-1}.
\]
\end{lem}

\begin{proof}
For $\left|w\right|=k$, each edge out of $w$ has conductance $a_{s}\left(w\right)/m^{k+1}$,
and there are $m$ such edges. Hence the total conductance contributed
by $w$ to $E_{k}$ is 
\[
m\cdot\frac{a_{s}\left(w\right)}{m^{k+1}}=\frac{a_{s}\left(w\right)}{m^{k}},
\]
and summing over $w\in W_{k}$ gives 
\[
C_{k}\left(s\right)=m^{-k}\sum_{w\in W_{k}}a_{s}\left(w\right).
\]
Inserting the identity from \prettyref{lem:telescoping} yields 
\[
C_{k}\left(s\right)=\left(u_{2k+1}\left(s\right)-u_{2k}\left(s\right)\right)/m^{k}.
\]

For the energy bound, fix $k\in\left\{ 0,\dots,N-1\right\} $. By
\prettyref{eq:2-8} and \prettyref{eq:2-9}, the total flow crossing
each cutset $E_{k}$ equals $1$, hence 
\[
\sum_{e\in E_{k}}\theta\left(e\right)=1.
\]
By Cauchy-Schwarz, 
\begin{gather*}
1=\left(\sum_{e\in E_{k}}\theta\left(e\right)\right)^{2}\le\left(\sum_{e\in E_{k}}\frac{\theta\left(e\right)^{2}}{c_{s}\left(e\right)}\right)\left(\sum_{e\in E_{k}}c_{s}\left(e\right)\right)=\left(\sum_{e\in E_{k}}\frac{\theta\left(e\right)^{2}}{c_{s}\left(e\right)}\right)C_{k}\left(s\right)\\
\Updownarrow\\
\sum_{e\in E_{k}}\frac{\theta\left(e\right)^{2}}{c_{s}\left(e\right)}\ge\frac{1}{C_{k}\left(s\right)}.
\end{gather*}
Summing over $k=0,\dots,N-1$, using disjointness of the cutsets and
the definition \prettyref{eq:2-10} gives \prettyref{eq:2-15}. 
\end{proof}

\begin{rem}
The estimate \prettyref{eq:2-15} is the classical cutset lower bound
for the energy of a unit flow (a Nash-Williams-type inequality for
trees/electrical networks), specialized to the level cutsets $E_{k}$
in our canonical network. See e.g., \cite{MR920811,MR3616205}.
\end{rem}

Now we state a one-sided criterion for vanishing capacity. 
\begin{thm}
\label{thm:lower-bound}Fix $s\in X$. With $a_{s}\left(w\right)$,
$c_{s}\left(e\right)$, $C_{k}\left(s\right)$, $R_{N}\left(s\right)$,
and $\mathrm{Cap}_{N}\left(s\right)$ as above, the following hold.
\begin{enumerate}
\item For every $N\ge1$, 
\begin{equation}
R_{N}\left(s\right)\ge\sum^{N-1}_{k=0}\frac{m^{k}}{u_{2k+1}\left(s\right)-u_{2k}\left(s\right)},\label{eq:2-16}
\end{equation}
with the convention that a zero denominator gives the corresponding
term $+\infty$. Equivalently, 
\begin{equation}
\mathrm{Cap}_{N}\left(s\right)\le\left(\sum^{N-1}_{k=0}\frac{m^{k}}{u_{2k+1}\left(s\right)-u_{2k}\left(s\right)}\right)^{-1}.\label{eq:2-17}
\end{equation}
\item Define 
\[
S\left(s\right):=\sum^{\infty}_{k=0}\frac{m^{k}}{u_{2k+1}\left(s\right)-u_{2k}\left(s\right)}\in\left(0,\infty\right].
\]
If $S\left(s\right)=\infty$ then $R_{N}\left(s\right)\to\infty$
and $\mathrm{Cap}_{N}\left(s\right)\to0$ as $N\to\infty$. If $S\left(s\right)<\infty$,
the bound above does not, by itself, imply a positive lower bound
on $\mathrm{Cap}_{N}\left(s\right)$.
\end{enumerate}
\end{thm}

\begin{proof}
The bound \prettyref{eq:2-16} (equivalently, \prettyref{eq:2-17})
is \prettyref{lem:cutset} with the explicit formula for $C_{k}\left(s\right)$. 

For part (2), write 
\[
S_{N}\left(s\right):=\sum^{N-1}_{k=0}\frac{m^{k}}{u_{2k+1}\left(s\right)-u_{2k}\left(s\right)}.
\]
Then $R_{N}\left(s\right)\ge S_{N}\left(s\right)$. If $S\left(s\right)=\infty$
then $S_{N}\left(s\right)\to\infty$, hence $R_{N}\left(s\right)\to\infty$
and $\mathrm{Cap}_{N}\left(s\right)=1/R_{N}\left(s\right)\to0$. 
\end{proof}

To get an upper bound, we need a certain concentration condition;
see \prettyref{eq:2-18}. Under this, a uniform splitting flow has
energy within a fixed factor of the lower bound. 

Specifically, assume there exists $\Lambda\ge1$ such that 
\begin{equation}
\left(\sum_{w\in W_{k}}a_{s}\left(w\right)\right)\left(\sum_{w\in W_{k}}\frac{1}{a_{s}\left(w\right)}\right)\le\Lambda m^{2k}\label{eq:2-18}
\end{equation}
for every $k\ge0$ with $\sum_{w\in W_{k}}a_{s}\left(w\right)>0$.

If $\sum_{w\in W_{k}}a_{s}\left(w\right)=0$ for some $k$, then $C_{k}\left(s\right)=0$
and $R_{N}\left(s\right)=+\infty$ for all $N\ge k+1$, and the comparisons
below hold with $+\infty$ on both sides.
\begin{thm}
\label{thm:upperbound}Under \prettyref{eq:2-18}, for every $N\ge1$,
\begin{equation}
\sum^{N-1}_{k=0}\frac{m^{k}}{u_{2k+1}\left(s\right)-u_{2k}\left(s\right)}\le R_{N}\left(s\right)\le\Lambda\sum^{N-1}_{k=0}\frac{m^{k}}{u_{2k+1}\left(s\right)-u_{2k}\left(s\right)}.\label{eq:2-19}
\end{equation}
Equivalently, 
\begin{equation}
\frac{1}{\Lambda}\left(\sum^{N-1}_{k=0}\frac{m^{k}}{u_{2k+1}\left(s\right)-u_{2k}\left(s\right)}\right)^{-1}\le\mathrm{Cap}_{N}\left(s\right)\le\left(\sum^{N-1}_{k=0}\frac{m^{k}}{u_{2k+1}\left(s\right)-u_{2k}\left(s\right)}\right)^{-1}.\label{eq:2-20}
\end{equation}
\end{thm}

\begin{proof}
The lower bound in \prettyref{eq:2-19} is from \prettyref{thm:lower-bound}.
For the upper bound, it suffices to give a unit flow with controlled
energy. Define the uniform splitting flow by 
\[
\theta^{\mathrm{unif}}\left(w,wi\right):=\frac{1}{m^{\left|w\right|+1}}\qquad\left(\left|w\right|\le N-1\right).
\]
This is a unit flow in the sense of \prettyref{eq:2-8}-\prettyref{eq:2-9},
since $\sum^{m}_{i=1}\theta^{\mathrm{unif}}\left(\emptyset,i\right)=1$
and, for $0<\left|w\right|<N$, 
\[
\theta^{\mathrm{unif}}\left(w^{-},w\right)=\frac{1}{m^{\left|w\right|}}=\sum^{m}_{i=1}\frac{1}{m^{\left|w\right|+1}}=\sum^{m}_{i=1}\theta^{\mathrm{unif}}\left(w,wi\right).
\]
Fix $k\in\left\{ 0,\dots,N-1\right\} $. For an edge $e=\left(w,wi\right)$
with $\left|w\right|=k$ we have 
\[
\theta^{\mathrm{unif}}\left(e\right)=\frac{1}{m^{k+1}},\qquad c_{s}\left(e\right)=\frac{a_{s}\left(w\right)}{m^{k+1}},
\]
hence 
\[
\frac{\theta^{\mathrm{unif}}\left(e\right)^{2}}{c_{s}\left(e\right)}=\frac{1/m^{2k+2}}{a_{s}\left(w\right)/m^{k+1}}=\frac{1}{m^{k+1}}\cdot\frac{1}{a_{s}\left(w\right)}.
\]
Summing over the $m$ children of $w$ gives 
\[
\sum^{m}_{i=1}\frac{\theta^{\mathrm{unif}}\left(w,wi\right)^{2}}{c_{s}\left(w,wi\right)}=\frac{1}{m^{k}}\cdot\frac{1}{a_{s}\left(w\right)}.
\]
So the level-$k$ contribution is 
\[
\sum_{e\in E_{k}}\frac{\theta^{\mathrm{unif}}\left(e\right)^{2}}{c_{s}\left(e\right)}=\frac{1}{m^{k}}\sum_{w\in W_{k}}\frac{1}{a_{s}\left(w\right)}.
\]
By the \prettyref{eq:2-18}, 
\[
\sum_{w\in W_{k}}\frac{1}{a_{s}\left(w\right)}\le\Lambda\frac{m^{2k}}{\sum_{w\in W_{k}}a_{s}\left(w\right)}.
\]
Hence 
\[
\sum_{e\in E_{k}}\frac{\theta^{\mathrm{unif}}\left(e\right)^{2}}{c_{s}\left(e\right)}\le\frac{1}{m^{k}}\cdot\Lambda\frac{m^{2k}}{\sum_{w\in W_{k}}a_{s}\left(w\right)}=\Lambda\frac{m^{k}}{\sum_{w\in W_{k}}a_{s}\left(w\right)}=\Lambda\frac{1}{C_{k}\left(s\right)}.
\]
Summing over $k=0,\dots,N-1$ yields 
\[
\mathcal{E}_{s}\left(\theta^{\mathrm{unif}}\right)\le\Lambda\sum^{N-1}_{k=0}\frac{1}{C_{k}\left(s\right)}=\Lambda\sum^{N-1}_{k=0}\frac{m^{k}}{u_{2k+1}\left(s\right)-u_{2k}\left(s\right)}.
\]
Since $R_{N}\left(s\right)$ is the infimum of $\mathcal{E}_{s}\left(\theta\right)$
over all unit flows, we get the upper bound in \prettyref{eq:2-19}.
And \prettyref{eq:2-20} is immediate. 
\end{proof}

\begin{cor}
\label{cor:capacity-criterion} Fix $s\in X$ and assume \prettyref{eq:2-18}
holds at $s$ with constant $\Lambda$. Set 
\begin{equation}
S\left(s\right):=\sum^{\infty}_{k=0}\frac{m^{k}}{u_{2k+1}\left(s\right)-u_{2k}\left(s\right)}\in\left(0,\infty\right].\label{eq:2-21}
\end{equation}
Then the limit 
\[
R_{\infty}\left(s\right):=\lim_{N\to\infty}R_{N}\left(s\right)\in\left(0,\infty\right]
\]
exists and satisfies 
\[
S\left(s\right)\le R_{\infty}\left(s\right)\le\Lambda S\left(s\right).
\]
Equivalently, 
\[
\frac{1}{\Lambda S\left(s\right)}\le\mathrm{Cap}_{\infty}\left(s\right)\le\frac{1}{S\left(s\right)}.
\]
In particular, $\mathrm{Cap}_{\infty}\left(s\right)>0$ if and only
if $S\left(s\right)<\infty$. 
\end{cor}

\begin{proof}
If $N'\ge N$, every unit flow $\theta'$ to depth $N'$ restricts
to a unit flow $\theta$ to depth $N$ with $\mathcal{E}_{s}\left(\theta\right)\le\mathcal{E}_{s}\left(\theta'\right)$.
Consequently $R_{N}\left(s\right)\le R_{N'}\left(s\right)$, so the
limit $R_{\infty}\left(s\right)$ exists in $\left(0,\infty\right]$.
Define $\mathrm{Cap}_{\infty}\left(s\right):=1/R_{\infty}\left(s\right)\in\left[0,\infty\right)$.

Now write 
\[
S_{N}\left(s\right):=\sum^{N-1}_{k=0}\frac{m^{k}}{u_{2k+1}\left(s\right)-u_{2k}\left(s\right)}.
\]
By Theorems \ref{thm:lower-bound} and \ref{thm:upperbound}, for
every $N\ge1$, 
\[
S_{N}\left(s\right)\le R_{N}\left(s\right)\le\Lambda S_{N}\left(s\right).
\]
Since $S_{N}\left(s\right)\uparrow S\left(s\right)$ as $N\to\infty$,
passing to the limit gives 
\begin{gather*}
S\left(s\right)\le R_{\infty}\left(s\right)\le\Lambda S\left(s\right)\\
\Updownarrow\\
\frac{1}{\Lambda S\left(s\right)}\le\mathrm{Cap}_{\infty}\left(s\right)\le\frac{1}{S\left(s\right)}
\end{gather*}
In particular, $\mathrm{Cap}_{\infty}\left(s\right)>0$ holds  when
$S\left(s\right)<\infty$. 
\end{proof}

\section{Invariant completion}\label{sec:3}

We now pass from diagonal control to an off-diagonal completion statement.
Throughout we keep the setting and notation of the preceding section:
$K:X\times X\to\mathbb{C}$ is positive definite, $L$ is the pullback
operator 
\[
\left(LJ\right)\left(s,t\right):=\sum^{m}_{i=1}J\left(\varphi_{i}\left(s\right),\varphi_{i}\left(t\right)\right),
\]
and we assume $LK\ge K$. We write $K_{0}:=K$ and $K_{n+1}:=LK_{n}$
for $n\ge0$, and $u_{n}\left(s\right):=K_{n}\left(s,s\right)$. For
each $s\in X$ the monotone limit 
\[
u_{\infty}\left(s\right):=\lim_{n\to\infty}u_{n}\left(s\right)\in\left[0,\infty\right]
\]
exists. 

We set 
\[
X_{\mathrm{fin}}:=\left\{ s\in X:u_{\infty}\left(s\right)<\infty\right\} .
\]
On $X_{\mathrm{fin}}$ the diagonals are controlled, and the aim is
to show that this control already yields convergence of the full kernel
tower on $X_{\mathrm{fin}}\times X_{\mathrm{fin}}$. The idea is that
positivity propagates diagonal bounds to off-diagonal increments. 
\begin{lem}
\label{lem:3-1}Let $J:X\times X\to\mathbb{C}$ be positive definite.
Then for every $s,t\in X$, 
\[
\left|J\left(s,t\right)\right|^{2}\le J\left(s,s\right)J\left(t,t\right).
\]
\end{lem}

\begin{proof}
Positivity of $J$ implies that the Gram matrix 
\[
\left[\begin{matrix}J\left(s,s\right) & J\left(s,t\right)\\
J\left(t,s\right) & J\left(t,t\right)
\end{matrix}\right]
\]
is positive semidefinite. In particular its determinant is nonnegative,
which is the claim.
\end{proof}

We now apply \prettyref{lem:3-1} to the increments 
\[
D_{n}:=K_{n+1}-K_{n}.
\]
This gives a uniform control of off-diagonal increments by diagonal
increments. 
\begin{cor}
The following pointwise estimate holds:
\begin{equation}
\left|K_{n+1}\left(s,t\right)-K_{n}\left(s,t\right)\right|\leq\left(u_{n+1}\left(s\right)-u_{n}\left(s\right)\right)^{1/2}\left(u_{n+1}\left(t\right)-u_{n}\left(t\right)\right)^{1/2}.\label{eq:3-1}
\end{equation}
\end{cor}

\begin{proof}
Since $K_{n+1}=LK_{n}\ge K_{n}$, each $D_{n}$ is positive definite.
Moreover, 
\[
D_{n}\left(s,s\right)=u_{n+1}\left(s\right)-u_{n}\left(s\right),\quad s\in X.
\]
\prettyref{eq:3-1} follows from \prettyref{lem:3-1} with $J=D_{n}$.
\end{proof}

\begin{lem}
\label{lem:3-3}Fix $s,t\in X_{\mathrm{fin}}$. Then the series 
\[
\sum^{\infty}_{n=0}\left|K_{n+1}\left(s,t\right)-K_{n}\left(s,t\right)\right|
\]
converges. In particular, the limit 
\[
K_{\infty}\left(s,t\right):=\lim_{n\to\infty}K_{n}\left(s,t\right)
\]
exists as a finite complex number. 
\end{lem}

\begin{proof}
By \prettyref{eq:3-1}, 
\[
\sum^{N-1}_{n=0}\left|K_{n+1}\left(s,t\right)-K_{n}\left(s,t\right)\right|\le\sum^{N-1}_{n=0}\left(u_{n+1}\left(s\right)-u_{n}\left(s\right)\right)^{1/2}\left(u_{n+1}\left(t\right)-u_{n}\left(t\right)\right)^{1/2}.
\]
By Cauchy-Schwarz, 
\begin{multline*}
\sum^{N-1}_{n=0}\left(u_{n+1}\left(s\right)-u_{n}\left(s\right)\right)^{1/2}\left(u_{n+1}\left(t\right)-u_{n}\left(t\right)\right)^{1/2}\\
\le\left(\sum^{N-1}_{n=0}\left(u_{n+1}\left(s\right)-u_{n}\left(s\right)\right)\right)^{1/2}\left(\sum^{N-1}_{n=0}\left(u_{n+1}\left(t\right)-u_{n}\left(t\right)\right)\right)^{1/2}.
\end{multline*}
Then we get
\[
\sum^{N-1}_{n=0}\left(u_{n+1}\left(s\right)-u_{n}\left(s\right)\right)=u_{N}\left(s\right)-u_{0}\left(s\right)\le u_{\infty}\left(s\right)-u_{0}\left(s\right)<\infty,
\]
and similarly for $t$. Letting $N\to\infty$ shows that the series
of absolute increments converges, hence $\left(K_{n}\left(s,t\right)\right)$
is Cauchy and has a finite limit. 
\end{proof}

\prettyref{lem:3-3} provides pointwise convergence on $X_{\mathrm{fin}}\times X_{\mathrm{fin}}$.
We now verify that the limit kernel inherits the structural properties
of the sequence $K_{n}$.
\begin{thm}
\label{thm:3-4}There exists a unique kernel $K_{\infty}:X_{\mathrm{fin}}\times X_{\mathrm{fin}}\to\mathbb{C}$
with the following properties.
\begin{enumerate}
\item For every $s,t\in X_{\mathrm{fin}}$, one has 
\[
K_{\infty}\left(s,t\right)=\lim_{n\to\infty}K_{n}\left(s,t\right).
\]
\item $K_{\infty}$ is positive definite on $X_{\mathrm{fin}}$.
\item $K_{\infty}$ is $L$-invariant on $X_{\mathrm{fin}}$ in the sense
that for every $s,t\in X_{\mathrm{fin}}$, 
\[
\left(LK_{\infty}\right)\left(s,t\right)=K_{\infty}\left(s,t\right),
\]
and $K_{\infty}\ge K$ on $X_{\mathrm{fin}}$ in the kernel order.
\item If $J$ is any kernel on $X_{\mathrm{fin}}$ such that $J\ge K$ and
$LJ=J$, then $J\ge K_{\infty}$ on $X_{\mathrm{fin}}$ in the kernel
order.
\end{enumerate}
\end{thm}

\begin{proof}
Existence and uniqueness of the pointwise limit $K_{\infty}\left(s,t\right)$
for $s,t\in X_{\mathrm{fin}}$ is given by \prettyref{lem:3-3}.

To prove positive definiteness, fix points $s_{1},\dots,s_{r}\in X_{\mathrm{fin}}$
and consider the Gram matrices 
\[
G_{n}:=\left[K_{n}\left(s_{i},s_{j}\right)\right]^{r}_{i,j=1}.
\]
Each $G_{n}$ is positive semidefinite because $K_{n}$ is positive
definite. Moreover, for each $i,j$ the limit $\lim_{n\to\infty}K_{n}\left(s_{i},s_{j}\right)$
exists by \prettyref{lem:3-3}, hence $G_{n}$ converges entrywise
to the matrix 
\[
G_{\infty}:=\left[K_{\infty}\left(s_{i},s_{j}\right)\right]^{r}_{i,j=1}.
\]
For any vector $c\in\mathbb{C}^{r}$ we have $c^{*}G_{n}c\ge0$ for
all $n$, and $c^{*}G_{n}c\to c^{*}G_{\infty}c$ by entrywise convergence.
Therefore $c^{*}G_{\infty}c\ge0$ for all $c$, so $G_{\infty}$ is
positive semidefinite. Since the choice of $\left(s_{i}\right)$ was
arbitrary, $K_{\infty}$ is positive definite on $X_{\mathrm{fin}}$.

To prove invariance, fix $s,t\in X_{\mathrm{fin}}$ and write 
\[
K_{n+1}\left(s,t\right)=\left(LK_{n}\right)\left(s,t\right)=\sum^{m}_{i=1}K_{n}\left(\varphi_{i}\left(s\right),\varphi_{i}\left(t\right)\right).
\]
For each $i$ we have $\varphi_{i}\left(s\right),\varphi_{i}\left(t\right)\in X_{\mathrm{fin}}$.
Indeed, for every $n\ge0$ the diagonal recursion gives 
\[
u_{n+1}\left(s\right)=\sum^{m}_{j=1}u_{n}\left(\varphi_{j}\left(s\right)\right),
\]
so $u_{n}\left(\varphi_{i}\left(s\right)\right)\le u_{n+1}\left(s\right)\le u_{\infty}\left(s\right)$.
Letting $n\to\infty$ gives $u_{\infty}\left(\varphi_{i}\left(s\right)\right)\le u_{\infty}\left(s\right)<\infty$,
and similarly for $t$. 

Thus \prettyref{lem:3-3} applies at the points $\varphi_{i}\left(s\right),\varphi_{i}\left(t\right)$,
and we may pass to the limit termwise in the finite sum, so that
\begin{align*}
\lim_{n\to\infty}K_{n+1}\left(s,t\right) & =\sum^{m}_{i=1}\lim_{n\to\infty}K_{n}\left(\varphi_{i}\left(s\right),\varphi_{i}\left(t\right)\right)\\
 & =\sum^{m}_{i=1}K_{\infty}\left(\varphi_{i}\left(s\right),\varphi_{i}\left(t\right)\right)=\left(LK_{\infty}\right)\left(s,t\right).
\end{align*}
And the left-hand side is $K_{\infty}\left(s,t\right)$.

Since $K_{n+1}=LK_{n}\ge K_{n}$ and $K_{0}=K$, we have $K_{n}\ge K$
for all $n$ in the kernel order, hence $K_{\infty}\ge K$ on $X_{\mathrm{fin}}$
by passing to limits. 

For the minimality statement, let $J$ be a kernel on $X_{\mathrm{fin}}$
with $J\ge K$ and $LJ=J$. Then for every $n\ge0$, 
\[
J=L^{n}J\ge L^{n}K=K_{n}
\]
in the kernel order. Fix finitely many points $s_{1},\dots,s_{r}\in X_{\mathrm{fin}}$.
The corresponding Gram matrices satisfy $J\left[s_{1},\dots,s_{r}\right]\ge G_{n}$
for all $n$, and since $G_{n}\to G_{\infty}$ entrywise we obtain
$J\left[s_{1},\dots,s_{r}\right]\ge G_{\infty}$. Since the points
were arbitrary, $J\ge K_{\infty}$ on $X_{\mathrm{fin}}$.
\end{proof}

We include two consequences to connect the set $X_{\mathrm{fin}}$
to the resistance setting in \prettyref{sec:2}. (i) If $s\in X_{\mathrm{fin}}$
then the associated canonical network at $s$ has infinite effective
resistance to infinity, i.e.,$R_{N}\left(s\right)\to+\infty$. (ii)
There is a levelwise localization condition on the diagonal increments
that is sufficient for $s\in X_{\mathrm{fin}}$. The idea is that
\prettyref{lem:telescoping} converts levelwise increment mass into
diagonal growth, so summability yields the monotone sequence $\left(u_{n}\left(s\right)\right)$
to have a finite limit. 
\begin{prop}
\label{prop:3-5} Fix $s\in X_{\mathrm{fin}}$. Then 
\[
S\left(s\right)=\sum^{\infty}_{k=0}\frac{m^{k}}{u_{2k+1}\left(s\right)-u_{2k}\left(s\right)}=+\infty,
\]
with the convention that a zero denominator contributes $+\infty$.
Consequently, $R_{N}\left(s\right)\to+\infty$ and $\mathrm{Cap}_{N}\left(s\right)\to0$
as $N\to\infty$, hence $\mathrm{Cap}_{\infty}\left(s\right)=0$. 
\end{prop}

\begin{proof}
Set 
\[
d_{n}:=u_{n+1}\left(s\right)-u_{n}\left(s\right)\ge0.
\]
Since $s\in X_{\mathrm{fin}}$, the monotone limit $u_{\infty}\left(s\right)<\infty$
exists, hence 
\[
\sum^{\infty}_{n=0}d_{n}=\sum^{\infty}_{n=0}\left(u_{n+1}\left(s\right)-u_{n}\left(s\right)\right)=u_{\infty}\left(s\right)-u_{0}\left(s\right)<\infty.
\]
Therefore $d_{n}\to0$ as $n\to\infty$, and in particular $d_{2k}\to0$
as $k\to\infty$.

If $d_{2k}=u_{2k+1}\left(s\right)-u_{2k}\left(s\right)=0$ for some
$k$, then by convention the corresponding term in $S\left(s\right)$
is $+\infty$, hence $S\left(s\right)=+\infty$. Otherwise $d_{2k}>0$
for all $k$, and since $m^{k}\ge1$ we have 
\[
\frac{m^{k}}{u_{2k+1}\left(s\right)-u_{2k}\left(s\right)}=\frac{m^{k}}{d_{2k}}\ge\frac{1}{d_{2k}}\to+\infty.
\]
In particular, the terms of the defining series for $S\left(s\right)$
do not tend to $0$, so $S\left(s\right)=+\infty$.

Finally, \prettyref{thm:lower-bound} gives for every $N\ge1$, 
\[
R_{N}\left(s\right)\ge\sum^{N-1}_{k=0}\frac{m^{k}}{u_{2k+1}\left(s\right)-u_{2k}\left(s\right)}.
\]
Since the right-hand side diverges as $N\to\infty$, we get $R_{N}\left(s\right)\to+\infty$.
Therefore $\mathrm{Cap}_{N}\left(s\right)=1/R_{N}\left(s\right)\to0$,
and $\mathrm{Cap}_{\infty}\left(s\right)=0$. 
\end{proof}

\begin{rem}
\prettyref{prop:3-5} separates the two constructions. On $X_{\mathrm{fin}}$
the canonical tree network from \prettyref{sec:2} has vanishing limiting
capacity, i.e., 
\[
s\in X_{\mathrm{fin}}\Longrightarrow\mathrm{Cap}_{\infty}\left(s\right)=0.
\]
Thus the capacity estimates in \prettyref{sec:2} concern the case
where the diagonal sequence grows. Here we consider the other case,
where $u_{n}(s)$ has a finite limit. This is enough to define $K_{\infty}$
and obtain its minimality among $L$-invariant kernels dominating
$K$. 
\end{rem}

We end the section with a criterion for $s\in X_{\mathrm{fin}}$ in
terms of the level increments.
\begin{thm}
\label{thm:3-6} Fix $s\in X$. For each $k\ge0$ define 
\[
\Delta_{k}\left(s\right):=\sum_{w\in W_{k}}\left(u_{k+2}\left(\varphi_{w}\left(s\right)\right)-u_{k}\left(\varphi_{w}\left(s\right)\right)\right).
\]
Then 
\[
\Delta_{k}\left(s\right)=u_{2k+2}\left(s\right)-u_{2k}\left(s\right)\qquad\left(k\ge0\right),
\]
and, for every $N\ge1$, 
\[
u_{2N}\left(s\right)-u_{0}\left(s\right)=\sum^{N-1}_{k=0}\Delta_{k}\left(s\right).
\]
Consequently, 
\begin{equation}
s\in X_{\mathrm{fin}}\Longleftrightarrow\sum^{\infty}_{k=0}\Delta_{k}\left(s\right)<\infty.\label{eq:3-2}
\end{equation}
When these equivalent conditions hold, 
\begin{equation}
u_{\infty}\left(s\right)=u_{0}\left(s\right)+\sum^{\infty}_{k=0}\Delta_{k}\left(s\right).\label{eq:3-3}
\end{equation}
\end{thm}

\begin{proof}
For each $k\ge0$, expanding the diagonal recursion gives 
\[
\sum_{w\in W_{k}}u_{k+2}\left(\varphi_{w}\left(s\right)\right)=\sum_{w\in W_{k}}\sum_{v\in W_{k+2}}u_{0}\left(\varphi_{wv}\left(s\right)\right)=u_{2k+2}\left(s\right).
\]
Similarly, 
\[
\sum_{w\in W_{k}}u_{k}\left(\varphi_{w}\left(s\right)\right)=u_{2k}\left(s\right).
\]
Subtracting gives 
\[
\Delta_{k}\left(s\right)=u_{2k+2}\left(s\right)-u_{2k}\left(s\right).
\]
Therefore 
\[
\sum^{N-1}_{k=0}\Delta_{k}\left(s\right)=\sum^{N-1}_{k=0}\left(u_{2k+2}\left(s\right)-u_{2k}\left(s\right)\right)=u_{2N}\left(s\right)-u_{0}\left(s\right).
\]

Since $\left(u_{n}\left(s\right)\right)_{n\ge0}$ is increasing, $s\in X_{\mathrm{fin}}$
if and only if the even subsequence $\left(u_{2N}\left(s\right)\right)_{N\ge0}$
is bounded. The preceding identity shows that this is equivalent to
\prettyref{eq:3-2}. Letting $N\to\infty$ then gives \prettyref{eq:3-3}. 
\end{proof}

\begin{cor}
\label{cor:lq-localization-Xfin} Fix $s\in X$ and $q\in\left(1,\infty\right]$.
Suppose there is a sequence $\left(A_{k}\right)_{k\ge0}$ in $\left[0,\infty\right)$
such that for every $k\ge0$, 
\begin{equation}
m^{k\left(1-\frac{1}{q}\right)}\left(\sum_{w\in W_{k}}\left(u_{k+2}\left(\varphi_{w}\left(s\right)\right)-u_{k}\left(\varphi_{w}\left(s\right)\right)\right)^{q}\right)^{1/q}\le A_{k},\label{eq:3-4}
\end{equation}
with the usual interpretation for $q=\infty$, and assume 
\[
\sum^{\infty}_{k=0}A_{k}<\infty.
\]
Then $s\in X_{\mathrm{fin}}$ and 
\[
u_{\infty}\left(s\right)\le u_{0}\left(s\right)+\sum^{\infty}_{k=0}A_{k}.
\]
\end{cor}

\begin{proof}
By Hölder's inequality on $W_{k}$, 
\[
\Delta_{k}\left(s\right)\le m^{k\left(1-\frac{1}{q}\right)}\left(\sum_{w\in W_{k}}\left(u_{k+2}\left(\varphi_{w}\left(s\right)\right)-u_{k}\left(\varphi_{w}\left(s\right)\right)\right)^{q}\right)^{1/q}\le A_{k}.
\]
Thus 
\[
\sum^{\infty}_{k=0}\Delta_{k}\left(s\right)\le\sum^{\infty}_{k=0}A_{k}<\infty.
\]
The conclusion follows from \prettyref{thm:3-6} and \prettyref{eq:3-3}. 
\end{proof}

\section{Splitting isometry on $\mathcal{H}\left(K_{\infty}\right)$}\label{sec:4}

We now use the $L$-invariance of $K_{\infty}$ on its reproducing
kernel Hilbert space (RKHS). This gives an isometric splitting associated
with the maps $\left(\varphi_{i}\right)$, and its iterates are indexed
by words.

Let $\mathcal{H}_{\infty}:=\mathcal{H}\left(K_{\infty}\right)$ denote
the RKHS of $K_{\infty}$ on $X_{\mathrm{fin}}$. For each $s\in X_{\mathrm{fin}}$
write the kernel section
\[
k_{s}\left(\cdot\right):=K_{\infty}\left(\cdot,s\right)\in\mathcal{H}_{\infty}.
\]
Then $\left\langle f,k_{s}\right\rangle =f\left(s\right)$ for all
$f\in\mathcal{H}_{\infty}$, and 
\begin{align*}
\left\langle k_{s},k_{t}\right\rangle  & =K_{\infty}\left(s,t\right),\\
\left\Vert k_{s}\right\Vert ^{2} & =K_{\infty}\left(s,s\right)=u_{\infty}\left(s\right).
\end{align*}

\begin{lem}
\label{lem:4-1}For every $n\in\mathbb{N}_{0}$ and all $s,t\in X_{\mathrm{fin}}$,
\begin{equation}
K_{\infty}\left(s,t\right)=\sum_{w\in W_{n}}K_{\infty}\left(\varphi_{w}\left(s\right),\varphi_{w}\left(t\right)\right).\label{eq:4-1}
\end{equation}
\end{lem}

\begin{proof}
By \prettyref{thm:3-4} we have $LK_{\infty}=K_{\infty}$ on $X_{\mathrm{fin}}$,
hence $L^{n}K_{\infty}=K_{\infty}$ for every $n\ge0$. Expanding
$L^{n}$ gives, for $s,t\in X_{\mathrm{fin}}$, 
\[
\left(L^{n}K_{\infty}\right)\left(s,t\right)=\sum_{w\in W_{n}}K_{\infty}\left(\varphi_{w}\left(s\right),\varphi_{w}\left(t\right)\right),
\]
and this equals $K_{\infty}\left(s,t\right)$. 
\end{proof}

We now turn \prettyref{eq:4-1} into an operator identity on $\mathcal{H}_{\infty}$.
For each $n\in\mathbb{N}_{0}$ define the Hilbert space 
\[
\mathcal{H}^{\left(n\right)}_{\infty}:=\bigoplus_{w\in W_{n}}\mathcal{H}_{\infty},
\]
whose elements we write as vectors $\left(f_{w}\right)_{w\in W_{n}}$,
with inner product 
\[
\left\langle \left(f_{w}\right)_{w\in W_{n}},\left(g_{w}\right)_{w\in W_{n}}\right\rangle :=\sum_{w\in W_{n}}\left\langle f_{w},g_{w}\right\rangle .
\]
In particular, $\mathcal{H}^{\left(0\right)}_{\infty}:=\mathcal{H}_{\infty}$.
\begin{lem}
\label{lem:4-2}For each $n\in\mathbb{N}_{0}$ there exists a unique
linear map 
\[
V_{n}:\mathcal{H}_{\infty}\to\mathcal{H}^{\left(n\right)}_{\infty}
\]
such that for every $s\in X_{\mathrm{fin}}$, 
\begin{equation}
V_{n}k_{s}=\left(k_{\varphi_{w}\left(s\right)}\right)_{w\in W_{n}}.\label{eq:4-2}
\end{equation}
Moreover, $V_{n}$ is an isometry, hence $\left\Vert V_{n}f\right\Vert =\left\Vert f\right\Vert $
for all $f\in\mathcal{H}_{\infty}$. 
\end{lem}

\begin{proof}
Existence and uniqueness follow by linearity and density of the linear
span of $\left\{ k_{s}:s\in X_{\mathrm{fin}}\right\} $ in $\mathcal{H}_{\infty}$. 

To prove that $V_{n}$ is an isometry, it suffices to check preservation
of inner products on kernel vectors. Fix $s,t\in X_{\mathrm{fin}}$.
Using \prettyref{eq:4-2} and the definition of the direct sum inner
product, 
\[
\left\langle V_{n}k_{t},V_{n}k_{s}\right\rangle =\sum_{w\in W_{n}}\left\langle k_{\varphi_{w}\left(t\right)},k_{\varphi_{w}\left(s\right)}\right\rangle =\sum_{w\in W_{n}}K_{\infty}\left(\varphi_{w}\left(s\right),\varphi_{w}\left(t\right)\right).
\]

By \prettyref{eq:4-1} the last sum equals $K_{\infty}\left(s,t\right)$,
which in turn equals $\left\langle k_{t},k_{s}\right\rangle $. Hence
$\left\langle V_{n}k_{t},V_{n}k_{s}\right\rangle =\left\langle k_{t},k_{s}\right\rangle $
for all $s,t$. It follows that $\left\langle V_{n}f,V_{n}g\right\rangle =\left\langle f,g\right\rangle $
for all $f,g\in\mathcal{H}_{\infty}$, so $V_{n}$ is an isometry. 
\end{proof}

The case $n=1$ gives the basic splitting map 
\[
V:=V_{1}:\mathcal{H}_{\infty}\to\mathcal{H}^{\left(1\right)}_{\infty}=\mathcal{H}^{\oplus m}_{\infty}
\]
where 
\[
Vk_{s}=\left(k_{\varphi_{1}\left(s\right)},\dots,k_{\varphi_{m}\left(s\right)}\right).
\]
The adjoint $V^{*}:\mathcal{H}^{\oplus m}_{\infty}\to\mathcal{H}_{\infty}$
is a coisometry, specified as follows.
\begin{lem}
\label{lem:4-3}For $f=\left(f_{1},\dots,f_{m}\right)\in\mathcal{H}^{\oplus m}_{\infty}$
and $s\in X_{\mathrm{fin}}$,
\begin{equation}
\left(V^{*}f\right)\left(s\right)=\sum^{m}_{i=1}f_{i}\left(\varphi_{i}\left(s\right)\right).\label{eq:4-3}
\end{equation}
\end{lem}

\begin{proof}
By definition of the adjoint and the reproducing property, 
\[
\left\langle V^{*}f,k_{s}\right\rangle =\left\langle f,Vk_{s}\right\rangle =\left\langle \left(f_{i}\right)^{m}_{i=1},\left(k_{\varphi_{i}\left(s\right)}\right)^{m}_{i=1}\right\rangle =\sum^{m}_{i=1}\left\langle f_{i},k_{\varphi_{i}\left(s\right)}\right\rangle .
\]
Since $\left\langle f_{i},k_{x}\right\rangle =f_{i}\left(x\right)$
for $x\in X_{\mathrm{fin}}$, this is equivalent to \prettyref{eq:4-3}.
\end{proof}

We next relate the iterates $V_{n}$ to $V$ itself. Identify $\mathcal{H}^{\left(n+1\right)}_{\infty}$
with $(\mathcal{H}^{\left(n\right)}_{\infty})^{\oplus m}$ by grouping
coordinates according to the first letter: For each $v\in W_{n+1}$
write $v=iw$ with $i\in\left\{ 1,\dots,m\right\} $ and $w\in W_{n}$,
and regard a vector in $\mathcal{H}^{\left(n+1\right)}_{\infty}$
as $\left(F_{i}\right)^{m}_{i=1}$ with $F_{i}\in\mathcal{H}^{\left(n\right)}_{\infty}$
given by $\left(F_{i}\right)_{w}=f_{iw}$.
\begin{lem}
\label{lem:4-4}For every $n\ge0$ one has 
\begin{equation}
V_{n+1}=\left(V_{n}\oplus\cdots\oplus V_{n}\right)V,\label{eq:4-4}
\end{equation}
where $\left(V_{n}\oplus\cdots\oplus V_{n}\right)$ denotes the direct
sum of $m$ copies of $V_{n}$ acting from $\mathcal{H}^{\oplus m}_{\infty}$
to $(\mathcal{H}^{\left(n\right)}_{\infty})^{\oplus m}\cong\mathcal{H}^{\left(n+1\right)}_{\infty}$.
In particular, for every $s\in X_{\mathrm{fin}}$, 
\[
V_{n+1}k_{s}=\left(k_{\varphi_{v}\left(s\right)}\right)_{v\in W_{n+1}}.
\]
\end{lem}

\begin{proof}
It suffices to verify \prettyref{eq:4-4} on kernel vectors. Fix $s\in X_{\mathrm{fin}}$.
Then 
\[
Vk_{s}=\left(k_{\varphi_{i}\left(s\right)}\right)^{m}_{i=1}.
\]
Applying $\left(V_{n}\oplus\cdots\oplus V_{n}\right)$ gives the vector
whose $i$-th block equals 
\[
V_{n}k_{\varphi_{i}\left(s\right)}=\left(k_{\varphi_{w}\left(\varphi_{i}\left(s\right)\right)}\right)_{w\in W_{n}}=\left(k_{\varphi_{iw}\left(s\right)}\right)_{w\in W_{n}}.
\]
Under the identification $(\mathcal{H}^{\left(n\right)}_{\infty})^{\oplus m}\cong\mathcal{H}^{\left(n+1\right)}_{\infty}$
this is  $\left(k_{\varphi_{v}\left(s\right)}\right)_{v\in W_{n+1}}=V_{n+1}k_{s}$,
as claimed. 
\end{proof}

It remains to make explicit the word operators underlying the maps
$V_{n}$ and the corresponding Parseval identities. This gives a tree-indexed
packet system in $\mathcal{H}_{\infty}$.
\begin{lem}
\label{lem:4-5} Let $V=V_{1}:\mathcal{H}_{\infty}\to\mathcal{H}^{\oplus m}_{\infty}$
be the splitting isometry from above. For $i\in\left\{ 1,\dots,m\right\} $
let $S_{i}:\mathcal{H}_{\infty}\to\mathcal{H}_{\infty}$ be the coordinate
operators defined by 
\[
Vf=\left(S_{1}f,\dots,S_{m}f\right)\qquad\left(f\in\mathcal{H}_{\infty}\right).
\]
For a word $w=i_{1}\cdots i_{n}\in W_{n}$ set 
\[
S_{w}:=S_{i_{n}}\cdots S_{i_{1}}.
\]
Then for every $n\ge0$ and every $f\in\mathcal{H}_{\infty}$, 
\begin{equation}
V_{n}f=\left(S_{w}f\right)_{w\in W_{n}}.\label{eq:4-6}
\end{equation}
In particular, for every $s\in X_{\mathrm{fin}}$ and $w\in W_{n}$,
\[
S_{w}k_{s}=k_{\varphi_{w}\left(s\right)}.
\]
\end{lem}

\begin{proof}
By definition of $V$ and the splitting identity above, for $s\in X_{\mathrm{fin}}$
we have 
\[
Vk_{s}=\left(k_{\varphi_{1}\left(s\right)},\dots,k_{\varphi_{m}\left(s\right)}\right),
\]
so $S_{i}k_{s}=k_{\varphi_{i}\left(s\right)}$ for each $i$. By iteration,
for $w=i_{1}\cdots i_{n}\in W_{n}$ we obtain 
\[
S_{w}k_{s}=S_{i_{n}}\cdots S_{i_{1}}k_{s}=k_{\varphi_{i_{n}}\circ\cdots\circ\varphi_{i_{1}}\left(s\right)}=k_{\varphi_{w}\left(s\right)}.
\]
On the other hand, by definition of $V_{n}$ we have 
\[
V_{n}k_{s}=\left(k_{\varphi_{w}\left(s\right)}\right)_{w\in W_{n}}.
\]
Thus $V_{n}k_{s}=\left(S_{w}k_{s}\right)_{w\in W_{n}}$ for every
$s\in X_{\mathrm{fin}}$. Since the linear span of $\left\{ k_{s}:s\in X_{\mathrm{fin}}\right\} $
is dense in $\mathcal{H}_{\infty}$ and both sides depend linearly
and continuously on $f$, the identity \prettyref{eq:4-6} holds for
all $f\in\mathcal{H}_{\infty}$. 
\end{proof}

\begin{prop}
\label{prop:4-6}For every $n\ge0$ and every $f\in\mathcal{H}_{\infty}$,
\begin{equation}
\sum_{w\in W_{n}}\left\Vert S_{w}f\right\Vert ^{2}=\left\Vert f\right\Vert ^{2}.\label{eq:4-7}
\end{equation}
Equivalently, 
\begin{equation}
\sum_{w\in W_{n}}S^{*}_{w}S_{w}=I_{\mathcal{H}_{\infty}}.\label{eq:4-8}
\end{equation}
\end{prop}

\begin{proof}
Fix $n\ge0$. Using \prettyref{eq:4-6} and the definition of the
direct sum norm, 
\[
\left\Vert V_{n}f\right\Vert ^{2}=\sum_{w\in W_{n}}\left\Vert S_{w}f\right\Vert ^{2}\qquad\left(f\in\mathcal{H}_{\infty}\right).
\]
Since $V_{n}$ is an isometry by \prettyref{lem:4-2}, we have $\left\Vert V_{n}f\right\Vert =\left\Vert f\right\Vert $,
which gives \prettyref{eq:4-7}.

For \prettyref{eq:4-8}, set 
\[
A_{n}:=\sum_{w\in W_{n}}S^{*}_{w}S_{w}.
\]
Then $A_{n}$ is a positive operator and \prettyref{eq:4-7} implies
\[
\left\langle A_{n}f,f\right\rangle =\sum_{w\in W_{n}}\left\Vert S_{w}f\right\Vert ^{2}=\left\Vert f\right\Vert ^{2}=\left\langle f,f\right\rangle 
\]
for all $f\in\mathcal{H}_{\infty}$. Therefore 
\[
\left\langle \left(I_{\mathcal{H}_{\infty}}-A_{n}\right)f,f\right\rangle =0
\]
for all $f$. Since $I_{\mathcal{H}_{\infty}}-A_{n}$ is self-adjoint,
this forces $I_{\mathcal{H}_{\infty}}-A_{n}=0$, hence $A_{n}=I_{\mathcal{H}_{\infty}}$. 
\end{proof}

The splitting isometry $V:\mathcal{H}_{\infty}\to\mathcal{H}^{\oplus m}_{\infty}$
from \prettyref{lem:4-2} determines a tuple $\left(S_{1},\dots,S_{m}\right)$
on $\mathcal{H}_{\infty}$ satisfying $\sum^{m}_{i=1}S^{*}_{i}S_{i}=I_{\mathcal{H}_{\infty}}$.
This induces a unital completely positive map on $B\left(\mathcal{H}_{\infty}\right)$
given by $\Psi\left(X\right):=\sum^{m}_{i=1}S^{*}_{i}XS_{i}$. It
is natural to ask how the defect space $\mathcal{H}^{\oplus m}_{\infty}\ominus V\mathcal{H}_{\infty}$
and related decompositions of $V$ are reflected in the tree package
from \prettyref{sec:2}, and in particular how these operator-theoretic
features interact with the effective resistance and capacity of the
canonical network at $s$. We do not consider these questions here.

\section{Boundary martingales}\label{sec:5}

In \cite{tian2026sub} the boundary construction used the harmonic
majorant and its Doob transforms. Here we instead work on a single
product probability space $\Omega$. Sampling $K_{\infty}$ along
random words gives the martingales $\left(M_{n}\left(s,t;\cdot\right)\right)$
considered below.

Let $\Omega:=\left\{ 1,\dots,m\right\} ^{\mathbb{N}}$ with its product
$\sigma$-algebra, and let $\mu$ be the product probability measure
for which the coordinate maps are independent and uniformly distributed
on $\left\{ 1,\dots,m\right\} $. For $\omega=\omega_{1}\omega_{2}\cdots\in\Omega$
and $n\ge0$ we write 
\[
\omega|n:=\omega_{1}\cdots\omega_{n}\in W_{n},
\]
with the convention $\omega|0=\emptyset$. Let $\mathcal{F}_{n}$
denote the $\sigma$-algebra generated by the first $n$ coordinates,
so that $\left(\mathcal{F}_{n}\right)_{n\ge0}$ is the natural filtration
on $\Omega$.

Fix $s,t\in X_{\mathrm{fin}}$. For $n\ge0$ define the random variable
\begin{equation}
M_{n}\left(s,t;\omega\right):=m^{n}K_{\infty}\left(\varphi_{\omega\mid n}\left(s\right),\varphi_{\omega\mid n}\left(t\right)\right).\label{eq:5-1}
\end{equation}
Since $K_{\infty}$ is positive definite on $X_{\mathrm{fin}}$ and
$\varphi_{\omega\mid n}\left(s\right),\varphi_{\omega\mid n}\left(t\right)\in X_{\mathrm{fin}}$,
the quantity $M_{n}\left(s,t;\omega\right)$ is well-defined and finite
for every $\omega$.
\begin{lem}
\label{lem:5-1}For each fixed $s,t\in X_{\mathrm{fin}}$, the process
$\left(M_{n}\left(s,t;\cdot\right)\right)_{n\ge0}$ is a complex-valued
$\left(\mathcal{F}_{n}\right)$-martingale. Moreover, 
\begin{equation}
\mathbb{E}\left[M_{n}\left(s,t;\cdot\right)\right]=K_{\infty}\left(s,t\right)\qquad\left(n\ge0\right).\label{eq:5-2}
\end{equation}
\end{lem}

\begin{proof}
Fix $n\ge0$ and condition on the cylinder event $\left\{ \omega|n=w\right\} $.
Using the uniform distribution of $\omega_{n+1}$ and the definition
of $L$, 
\begin{align*}
\mathbb{E}\left[M_{n+1}\left(s,t;\cdot\right)\mid\omega|n=w\right] & =\frac{1}{m}\sum^{m}_{i=1}m^{n+1}K_{\infty}\left(\varphi_{wi}\left(s\right),\varphi_{wi}\left(t\right)\right)\\
 & =m^{n}\sum^{m}_{i=1}K_{\infty}\left(\varphi_{i}\left(\varphi_{w}\left(s\right)\right),\varphi_{i}\left(\varphi_{w}\left(t\right)\right)\right).
\end{align*}
By \prettyref{thm:3-4} we have $LK_{\infty}=K_{\infty}$ on $X_{\mathrm{fin}}\times X_{\mathrm{fin}}$,
hence 
\[
\sum^{m}_{i=1}K_{\infty}\left(\varphi_{i}\left(\varphi_{w}\left(s\right)\right),\varphi_{i}\left(\varphi_{w}\left(t\right)\right)\right)=K_{\infty}\left(\varphi_{w}\left(s\right),\varphi_{w}\left(t\right)\right).
\]
Therefore, on $\left\{ \omega|n=w\right\} $, we have
\[
\mathbb{E}\left[M_{n+1}\left(s,t;\cdot\right)\mid\omega|n=w\right]=m^{n}K_{\infty}\left(\varphi_{w}\left(s\right),\varphi_{w}\left(t\right)\right)=M_{n}\left(s,t;\omega\right).
\]
This shows $\mathbb{E}\left[M_{n+1}\left(s,t;\cdot\right)\,\middle|\,\mathcal{F}_{n}\right]=M_{n}\left(s,t;\cdot\right)$,
hence $\left(M_{n}\right)$ is a martingale.

To compute the mean, we expand the expectation over words of length
$n$, so that
\begin{align*}
\mathbb{E}\left[M_{n}\left(s,t;\cdot\right)\right] & =\sum_{w\in W_{n}}\mu\left(\omega|n=w\right)m^{n}K_{\infty}\left(\varphi_{w}\left(s\right),\varphi_{w}\left(t\right)\right)\\
 & =\sum_{w\in W_{n}}K_{\infty}\left(\varphi_{w}\left(s\right),\varphi_{w}\left(t\right)\right).
\end{align*}
Since $L^{n}K_{\infty}=K_{\infty}$, we have 
\[
\sum_{w\in W_{n}}K_{\infty}\left(\varphi_{w}\left(s\right),\varphi_{w}\left(t\right)\right)=\left(L^{n}K_{\infty}\right)\left(s,t\right)=K_{\infty}\left(s,t\right),
\]
which proves \prettyref{eq:5-2}. 
\end{proof}

The next step is to impose a condition guaranteeing uniform $L^{2}$
control. The relevant quantity is a levelwise square sum of the diagonal
values of $K_{\infty}$ along the $\varphi$-tree.

For $s\in X_{\mathrm{fin}}$ define 
\begin{equation}
B\left(s\right):=\sup_{n\ge0}\left\{ m^{n}\sum\nolimits_{w\in W_{n}}u_{\infty}\left(\varphi_{w}\left(s\right)\right)^{2}\right\} \in\left[0,\infty\right].\label{eq:5-3}
\end{equation}
We set 
\[
X_{2}:=\left\{ s\in X_{\mathrm{fin}}:B\left(s\right)<\infty\right\} .
\]
This condition holds under a convenient level regularity hypothesis
on the diagonal values $u_{\infty}\left(\varphi_{w}\left(s\right)\right)$.
We refer to \prettyref{eq:5-4} below as \emph{level comparability}
of the diagonal values along the $\varphi$-tree rooted at $s$. 
\begin{lem}
\label{lem:5-2-1} Fix $s\in X_{2}$ and $k\in\left\{ 1,\dots,m\right\} $.
Then $\varphi_{k}\left(s\right)\in X_{2}$ and 
\begin{equation}
B\left(\varphi_{k}\left(s\right)\right)\le\frac{1}{m}B\left(s\right).\label{eq:B-forward}
\end{equation}
\end{lem}

\begin{proof}
Fix $k\in\left\{ 1,\dots,m\right\} $ and $n\ge0$. By definition
\prettyref{eq:5-3}, 
\[
m^{n}\sum_{w\in W_{n}}u_{\infty}\left(\varphi_{w}\left(\varphi_{k}\left(s\right)\right)\right)^{2}=m^{n}\sum_{w\in W_{n}}u_{\infty}\left(\varphi_{kw}\left(s\right)\right)^{2}.
\]
Then 
\begin{align*}
m^{n}\sum_{w\in W_{n}}u_{\infty}\left(\varphi_{kw}\left(s\right)\right)^{2} & \le m^{n}\sum_{v\in W_{n+1}}u_{\infty}\left(\varphi_{v}\left(s\right)\right)^{2}\\
 & =\frac{1}{m}\left(m^{n+1}\sum_{v\in W_{n+1}}u_{\infty}\left(\varphi_{v}\left(s\right)\right)^{2}\right)\le\frac{1}{m}B\left(s\right).
\end{align*}
Taking the supremum over $n\ge0$ yields \prettyref{eq:B-forward}.
Since $s\in X_{2}$ implies $B\left(s\right)<\infty$, we conclude
$B\left(\varphi_{k}\left(s\right)\right)<\infty$, hence $\varphi_{k}\left(s\right)\in X_{2}$. 
\end{proof}

\begin{lem}
\label{lem:5-2}Fix $s\in X_{\mathrm{fin}}$ and suppose there exists
$\Gamma\ge1$ such that for every $n\ge0$, 
\begin{equation}
\max_{w\in W_{n}}u_{\infty}\left(\varphi_{w}\left(s\right)\right)\le\Gamma\min_{w\in W_{n}}u_{\infty}\left(\varphi_{w}\left(s\right)\right).\label{eq:5-4}
\end{equation}
Then $s\in X_{2}$ and 
\begin{equation}
B\left(s\right)\le\Gamma u_{\infty}\left(s\right)^{2}.\label{eq:B-bound}
\end{equation}
\end{lem}

\begin{proof}
Fix $n\ge0$ and write 
\begin{align*}
u^{\max}_{n}\left(s\right) & :=\max_{w\in W_{n}}u_{\infty}\left(\varphi_{w}\left(s\right)\right),\\
u^{\min}_{n}\left(s\right) & :=\min_{w\in W_{n}}u_{\infty}\left(\varphi_{w}\left(s\right)\right).
\end{align*}
We have $LK_{\infty}=K_{\infty}$, hence on the diagonal $u_{\infty}\left(x\right)=\sum^{m}_{i=1}u_{\infty}\left(\varphi_{i}\left(x\right)\right)$.
Iterating yields 
\begin{equation}
u_{\infty}\left(s\right)=\sum_{w\in W_{n}}u_{\infty}\left(\varphi_{w}\left(s\right)\right).\label{eq:5-6}
\end{equation}
Since there are $m^{n}$ terms in \prettyref{eq:5-6}, we have 
\[
m^{n}u^{\min}_{n}\left(s\right)\le u_{\infty}\left(s\right)\le m^{n}u^{\max}_{n}\left(s\right).
\]
Assumption \prettyref{eq:5-4} gives $u^{\max}_{n}\left(s\right)\le\Gamma u^{\min}_{n}\left(s\right)$,
hence 
\[
u^{\max}_{n}\left(s\right)\le\Gamma\frac{u_{\infty}\left(s\right)}{m^{n}}.
\]
Therefore, 
\begin{align*}
m^{n}\sum_{w\in W_{n}}u_{\infty}\left(\varphi_{w}\left(s\right)\right)^{2} & \le m^{n}\left(\max_{w\in W_{n}}u_{\infty}\left(\varphi_{w}\left(s\right)\right)\right)\sum_{w\in W_{n}}u_{\infty}\left(\varphi_{w}\left(s\right)\right)\\
 & \le m^{n}u^{\max}_{n}\left(s\right)u_{\infty}\left(s\right)\le\Gamma u_{\infty}\left(s\right)^{2}.
\end{align*}
Taking the supremum over $n$ yields \prettyref{eq:B-bound}. 
\end{proof}

\begin{rem}
The hypothesis \prettyref{eq:5-4} is a simple sufficient condition
ensuring $B\left(s\right)<\infty$. It is stronger than necessary.
Indeed, to conclude $s\in X_{2}$ one only needs the uniform bound
\[
\sup_{n\ge0}m^{n}\sum_{w\in W_{n}}u_{\infty}\left(\varphi_{w}\left(s\right)\right)^{2}<\infty,
\]
i.e., $B\left(s\right)<\infty$ itself. One may replace \prettyref{eq:5-4}
by any levelwise condition that yields a bound of the form 
\[
m^{n}\sum_{w\in W_{n}}u_{\infty}\left(\varphi_{w}\left(s\right)\right)^{2}\le Cu_{\infty}\left(s\right)^{2}\qquad\left(n\ge0\right),
\]
for some constant $C\ge1$. 
\end{rem}

We now prove that $B\left(s\right)$ controls the $L^{2}$ size of
the martingale $\left(M_{n}\left(s,t;\cdot\right)\right)$.
\begin{lem}
\label{lem:5-3}Fix $s,t\in X_{2}$. Then for every $n\ge0$, 
\begin{equation}
\mathbb{E}\left[\left|M_{n}\left(s,t;\cdot\right)\right|^{2}\right]\le B\left(s\right)^{1/2}B\left(t\right)^{1/2}.\label{eq:5-7}
\end{equation}
In particular, $\left(M_{n}\left(s,t;\cdot\right)\right)_{n\ge0}$
is bounded in $L^{2}\left(\Omega,\mu\right)$. 
\end{lem}

\begin{proof}
Fix $n\ge0$. By expanding the expectation over cylinders of length
$n$, 
\begin{align*}
\mathbb{E}\left[\left|M_{n}\left(s,t;\cdot\right)\right|^{2}\right] & =\sum_{w\in W_{n}}\mu\left(\omega\mid n=w\right)m^{2n}\left|K_{\infty}\left(\varphi_{w}\left(s\right),\varphi_{w}\left(t\right)\right)\right|^{2}\\
 & =m^{n}\sum_{w\in W_{n}}\left|K_{\infty}\left(\varphi_{w}\left(s\right),\varphi_{w}\left(t\right)\right)\right|^{2}.
\end{align*}
Applying \prettyref{lem:3-1} to $K_{\infty}$ gives 
\[
\left|K_{\infty}\left(\varphi_{w}\left(s\right),\varphi_{w}\left(t\right)\right)\right|^{2}\le u_{\infty}\left(\varphi_{w}\left(s\right)\right)u_{\infty}\left(\varphi_{w}\left(t\right)\right).
\]
Hence 
\[
\mathbb{E}\left[\left|M_{n}\left(s,t;\cdot\right)\right|^{2}\right]\le m^{n}\sum_{w\in W_{n}}u_{\infty}\left(\varphi_{w}\left(s\right)\right)u_{\infty}\left(\varphi_{w}\left(t\right)\right).
\]
By Cauchy-Schwarz in $\ell^{2}\left(W_{n}\right)$, 
\[
\sum_{w\in W_{n}}u_{\infty}\left(\varphi_{w}\left(s\right)\right)u_{\infty}\left(\varphi_{w}\left(t\right)\right)\le\left(\sum_{w\in W_{n}}u_{\infty}\left(\varphi_{w}\left(s\right)\right)^{2}\right)^{1/2}\left(\sum_{w\in W_{n}}u_{\infty}\left(\varphi_{w}\left(t\right)\right)^{2}\right)^{1/2}.
\]
Multiplying by $m^{n}$ and using the definition \prettyref{eq:5-3}
yields 
\begin{align*}
\mathbb{E}\left[\left|M_{n}\left(s,t;\cdot\right)\right|^{2}\right] & \le\left(m^{n}\sum_{w\in W_{n}}u_{\infty}\left(\varphi_{w}\left(s\right)\right)^{2}\right)^{1/2}\left(m^{n}\sum_{w\in W_{n}}u_{\infty}\left(\varphi_{w}\left(t\right)\right)^{2}\right)^{1/2}\\
 & \le B\left(s\right)^{1/2}B\left(t\right)^{1/2},
\end{align*}
which is \prettyref{eq:5-7}. 
\end{proof}

We can now pass to the boundary limit.
\begin{thm}
\label{thm:5-5}Fix $s,t\in X_{2}$. Then the martingale $\left(M_{n}\left(s,t;\cdot\right)\right)_{n\ge0}$
converges in $L^{2}\left(\Omega,\mu\right)$ and $\mu$-almost surely
to a limit $M_{\infty}\left(s,t;\cdot\right)\in L^{2}\left(\Omega,\mu\right)$.
Moreover, 
\begin{equation}
K_{\infty}\left(s,t\right)=\mathbb{E}\left[M_{\infty}\left(s,t;\cdot\right)\right].\label{eq:5-8}
\end{equation}
\end{thm}

\begin{proof}
By \prettyref{lem:5-1} the process $\left(M_{n}\left(s,t;\cdot\right)\right)$
is a martingale, and by \prettyref{lem:5-3} it is bounded in $L^{2}$.
The martingale convergence theorem in $L^{2}$ yields the existence
of a limit $M_{\infty}\left(s,t;\cdot\right)\in L^{2}$ such that
$M_{n}\to M_{\infty}$ in $L^{2}$ and $\mu$-almost surely.

Taking expectations in \prettyref{eq:5-2} and passing to the limit
gives 
\[
\mathbb{E}\left[M_{\infty}\left(s,t;\cdot\right)\right]=\lim_{n\to\infty}\mathbb{E}\left[M_{n}\left(s,t;\cdot\right)\right]=K_{\infty}\left(s,t\right),
\]
which is \prettyref{eq:5-8}. 
\end{proof}

For each fixed finite set of points, the limiting Gram matrix is positive
semidefinite almost everywhere.
\begin{thm}
\label{thm:5-6}Fix an integer $r\ge1$ and points $s_{1},\dots,s_{r}\in X_{2}$.
Then there exists a $\mu$-null set $N\left(s_{1},\dots,s_{r}\right)\subset\Omega$
such that for every $\omega\in\Omega\setminus N\left(s_{1},\dots,s_{r}\right)$
the matrix 
\[
\left[M_{\infty}\left(s_{i},s_{j};\omega\right)\right]^{r}_{i,j=1}
\]
is positive semidefinite. Moreover, for every $i,j\in\left\{ 1,\dots,r\right\} $,
\begin{equation}
K_{\infty}\left(s_{i},s_{j}\right)=\int_{\Omega}M_{\infty}\left(s_{i},s_{j};\omega\right)d\mu\left(\omega\right).\label{eq:5-9}
\end{equation}
\end{thm}

\begin{proof}
For each pair $\left(i,j\right)$, \prettyref{thm:5-5} gives a $\mu$-null
set $N_{i,j}\subset\Omega$ outside of which $M_{n}\left(s_{i},s_{j};\omega\right)\to M_{\infty}\left(s_{i},s_{j};\omega\right)$.
Set $N\left(s_{1},\dots,s_{r}\right):=\bigcup^{r}_{i,j=1}N_{i,j}$.

Fix $n\ge0$ and $\omega\in\Omega$. The matrix 
\[
G_{n}\left(\omega\right):=\left[M_{n}\left(s_{i},s_{j};\omega\right)\right]^{r}_{i,j=1}=m^{n}\left[K_{\infty}\left(\varphi_{\omega\mid n}\left(s_{i}\right),\varphi_{\omega\mid n}\left(s_{j}\right)\right)\right]^{r}_{i,j=1}
\]
is positive semidefinite, since it is a positive scalar multiple of
a Gram matrix of the positive definite kernel $K_{\infty}$. If $\omega\notin N\left(s_{1},\dots,s_{r}\right)$,
then each entry converges, hence $G_{n}\left(\omega\right)$ converges
entrywise to 
\[
G_{\infty}\left(\omega\right):=\left[M_{\infty}\left(s_{i},s_{j};\omega\right)\right]^{r}_{i,j=1}.
\]
Since the cone of positive semidefinite matrices is closed, $G_{\infty}\left(\omega\right)$
is positive semidefinite for every $\omega\notin N\left(s_{1},\dots,s_{r}\right)$.
This proves the first claim.

Finally, \prettyref{eq:5-9} is just \prettyref{eq:5-8} applied to
the pairs $\left(s_{i},s_{j}\right)$. 
\end{proof}

\begin{rem}
\label{rem:diag-L2} The condition $B\left(s\right)<\infty$ is stated
in terms of the diagonal function $u_{\infty}$, but the conclusion
of \prettyref{thm:5-6} is off-diagonal. 

For $s\in X_{\mathrm{fin}}$ and $n\ge0$ we have $M_{n}\left(s,s;\omega\right)=m^{n}u_{\infty}\left(\varphi_{\omega\mid n}\left(s\right)\right)\ge0$.
Moreover, 
\[
\mathbb{E}\left[\left|M_{n}\left(s,s;\cdot\right)\right|^{2}\right]=m^{n}\sum_{w\in W_{n}}u_{\infty}\left(\varphi_{w}\left(s\right)\right)^{2}.
\]
Consequently, 
\[
B\left(s\right)=\sup_{n\ge0}\mathbb{E}\left[\left|M_{n}\left(s,s;\cdot\right)\right|^{2}\right],
\]
so $X_{2}$ is exactly the set of points for which the \emph{diagonal}
martingale $\left(M_{n}\left(s,s;\cdot\right)\right)_{n\ge0}$ is
bounded in $L^{2}\left(\Omega,\mu\right)$. The hypothesis \prettyref{eq:5-4}
is a sufficient condition for this boundedness. 
\end{rem}

The $L^{2}$-bound also gives a square-function estimate for the martingale
differences. 
\begin{prop}
\label{prop:5-9} Fix $s,t\in X_{2}$ and set 
\[
D_{n}\left(s,t;\omega\right):=M_{n+1}\left(s,t;\omega\right)-M_{n}\left(s,t;\omega\right),\qquad n\ge0.
\]
Then 
\begin{multline}
\sum^{\infty}_{n=0}\mathbb{E}\left[\left|D_{n}\left(s,t;\cdot\right)\right|^{2}\right]\\
=\mathbb{E}\left[\left|M_{\infty}\left(s,t;\cdot\right)\right|^{2}\right]-\left|K_{\infty}\left(s,t\right)\right|^{2}\le B\left(s\right)^{1/2}B\left(t\right)^{1/2}-\left|K_{\infty}\left(s,t\right)\right|^{2}.\label{eq:5-11}
\end{multline}
In particular, the martingale square function 
\[
\left(\sum^{\infty}_{n=0}\left|D_{n}\left(s,t;\cdot\right)\right|^{2}\right)^{1/2}
\]
is in $L^{2}\left(\Omega,\mu\right)$ and its $L^{2}$-norm is bounded
by 
\[
\left(B\left(s\right)^{1/2}B\left(t\right)^{1/2}-\left|K_{\infty}\left(s,t\right)\right|^{2}\right)^{1/2}.
\]
\end{prop}

\begin{proof}
Fix $s,t\in X_{2}$. By \prettyref{thm:5-5}, the process $\left(M_{n}\left(s,t;\cdot\right)\right)_{n\ge0}$
is an $L^{2}$-martingale with limit $M_{\infty}\left(s,t;\cdot\right)$.
The martingale increments $D_{n}:=M_{n+1}-M_{n}$ are pairwise orthogonal
in $L^{2}\left(\Omega,\mu\right)$, hence for every $N\ge1$, 
\[
\mathbb{E}\left[\left|M_{N}\left(s,t;\cdot\right)-M_{0}\left(s,t;\cdot\right)\right|^{2}\right]=\sum^{N-1}_{n=0}\mathbb{E}\left[\left|D_{n}\left(s,t;\cdot\right)\right|^{2}\right].
\]
Since $M_{0}\left(s,t;\omega\right)=K_{\infty}\left(s,t\right)$ is
constant and $M_{N}\to M_{\infty}$ in $L^{2}$ as $N\to\infty$,
we obtain 
\[
\sum^{\infty}_{n=0}\mathbb{E}\left[\left|D_{n}\left(s,t;\cdot\right)\right|^{2}\right]=\mathbb{E}\left[\left|M_{\infty}\left(s,t;\cdot\right)-K_{\infty}\left(s,t\right)\right|^{2}\right].
\]
Expanding the right-hand side and using \prettyref{eq:5-8} gives
\[
\mathbb{E}\left[\left|M_{\infty}\left(s,t;\cdot\right)-K_{\infty}\left(s,t\right)\right|^{2}\right]=\mathbb{E}\left[\left|M_{\infty}\left(s,t;\cdot\right)\right|^{2}\right]-\left|K_{\infty}\left(s,t\right)\right|^{2},
\]
which yields the identity in \prettyref{eq:5-11}.

Finally, by \prettyref{lem:5-3} and convergence in $L^{2}$ we have
\[
\mathbb{E}\left[\left|M_{\infty}\left(s,t;\cdot\right)\right|^{2}\right]\le B\left(s\right)^{1/2}B\left(t\right)^{1/2},
\]
so the right-hand side of \prettyref{eq:5-11} is finite and the stated
$L^{2}$-bound for the square function follows. 
\end{proof}

\section{Shift cocycles and diagonal boundary factors}\label{sec:6}

We now state the pathwise identities satisfied by the boundary limits
constructed in the preceding section. These identities are given by
the word recursion and will be used later to move information along
the tree without returning to purely diagonal arguments.

Let $\sigma:\Omega\to\Omega$ denote the left shift 
\[
\sigma\left(\omega_{1}\omega_{2}\cdots\right):=\omega_{2}\omega_{3}\cdots,
\]
and for each $i\in\left\{ 1,\dots,m\right\} $ let $\iota_{i}:\Omega\to\Omega$
denote the prefix map $\iota_{i}\left(\omega\right):=i\omega$.

We begin with the exact cocycle at finite depth.
\begin{lem}
\label{lem:6-1}Fix $s,t\in X_{\mathrm{fin}}$. For every $n\ge0$,
every $i\in\left\{ 1,\dots,m\right\} $, and every $\omega\in\Omega$,
\begin{equation}
M_{n+1}\left(s,t;\iota_{i}\left(\omega\right)\right)=mM_{n}\left(\varphi_{i}\left(s\right),\varphi_{i}\left(t\right);\omega\right).\label{eq:6-1}
\end{equation}
Equivalently, for every $n\ge0$ and every $\omega\in\Omega$, 
\begin{equation}
M_{n+1}\left(s,t;\omega\right)=mM_{n}\left(\varphi_{\omega_{1}}\left(s\right),\varphi_{\omega_{1}}\left(t\right);\sigma\left(\omega\right)\right).\label{eq:6-2}
\end{equation}
\end{lem}

\begin{proof}
Fix $n\ge0$, $i\in\left\{ 1,\dots,m\right\} $, and $\omega\in\Omega$.
By definition \prettyref{eq:5-1}, 
\begin{align*}
M_{n+1}\left(s,t;\iota_{i}\left(\omega\right)\right) & =m^{n+1}K_{\infty}\left(\varphi_{\left(\iota_{i}\left(\omega\right)\right)\mid\left(n+1\right)}\left(s\right),\varphi_{\left(\iota_{i}\left(\omega\right)\right)\mid\left(n+1\right)}\left(t\right)\right)\\
 & =m^{n+1}K_{\infty}\left(\varphi_{\omega\mid n}\left(\varphi_{i}\left(s\right)\right),\varphi_{\omega\mid n}\left(\varphi_{i}\left(t\right)\right)\right)\\
 & =m\,m^{n}K_{\infty}\left(\varphi_{\omega\mid n}\left(\varphi_{i}\left(s\right)\right),\varphi_{\omega\mid n}\left(\varphi_{i}\left(t\right)\right)\right)\\
 & =mM_{n}\left(\varphi_{i}\left(s\right),\varphi_{i}\left(t\right);\omega\right),
\end{align*}
which is \prettyref{eq:6-1}. The identity \prettyref{eq:6-2} is
the same statement with $i=\omega_{1}$ and $\sigma\left(\omega\right)$
in place of $\omega$. 
\end{proof}

We now pass to the boundary limit. Throughout we keep the notation
of the preceding section. $X_{2}\subset X_{\mathrm{fin}}$ is the
set of points $s$ for which $B\left(s\right)<\infty$, and for $s,t\in X_{2}$
the martingale $\left(M_{n}\left(s,t;\cdot\right)\right)$ converges
in $L^{2}$ and $\mu$-almost surely to $M_{\infty}\left(s,t;\cdot\right)$
by \prettyref{thm:5-5}.

Recall that, by \prettyref{lem:5-2-1}, $X_{2}$ is stable under the
maps $\varphi_{k}$, so the boundary limits $M_{\infty}\left(\varphi_{k}\left(s\right),\varphi_{k}\left(t\right);\cdot\right)$
exist whenever $s,t\in X_{2}$.
\begin{thm}
\label{thm:6-2}Fix $r\ge1$ and points $s_{1},\dots,s_{r}\in X_{2}$.
There exists a $\mu$-null set $N\left(s_{1},\dots,s_{r}\right)\subset\Omega$
with the following property:

For every $\omega\in\Omega\setminus N\left(s_{1},\dots,s_{r}\right)$
and every $k\in\left\{ 1,\dots,m\right\} $, 
\begin{equation}
M_{\infty}\left(s_{i},s_{j};\iota_{k}\left(\omega\right)\right)=mM_{\infty}\left(\varphi_{k}\left(s_{i}\right),\varphi_{k}\left(s_{j}\right);\omega\right)\qquad\left(1\le i,j\le r\right).\label{eq:6-3}
\end{equation}
In particular, for every $\omega\in\Omega\setminus N\left(s_{1},\dots,s_{r}\right)$
and every $i,j\in\left\{ 1,\dots,r\right\} $, one has 
\begin{equation}
M_{\infty}\left(s_{i},s_{j};\omega\right)=mM_{\infty}\left(\varphi_{\omega_{1}}\left(s_{i}\right),\varphi_{\omega_{1}}\left(s_{j}\right);\sigma\left(\omega\right)\right).\label{eq:6-4}
\end{equation}
\end{thm}

\begin{proof}
Fix $1\le i,j\le r$. By \prettyref{thm:5-5} applied to the pairs
\[
\left(s_{i},s_{j}\right)\;\text{and }\left(\varphi_{k}\left(s_{i}\right),\varphi_{k}\left(s_{j}\right)\right)
\]
for $k\in\left\{ 1,\dots,m\right\} $, choose $\mu$-null sets on
which the almost sure limits defining $M_{\infty}$ fail. In addition,
for each $k$ we also require the almost sure limits for $\left(s_{i},s_{j}\right)$
to hold at $\iota_{k}\left(\omega\right)$. Taking the union over
the finitely many indices produces a null set $N\left(s_{1},\dots,s_{r}\right)$
such that for every $\omega\in\Omega\setminus N\left(s_{1},\dots,s_{r}\right)$,
every $k\in\left\{ 1,\dots,m\right\} $, and every $1\le i,j\le r$,
all limits 
\begin{gather*}
M_{n}\left(s_{i},s_{j};\iota_{k}\left(\omega\right)\right)\to M_{\infty}\left(s_{i},s_{j};\iota_{k}\left(\omega\right)\right),\\
M_{n}\left(\varphi_{k}\left(s_{i}\right),\varphi_{k}\left(s_{j}\right);\omega\right)\to M_{\infty}\left(\varphi_{k}\left(s_{i}\right),\varphi_{k}\left(s_{j}\right);\omega\right)
\end{gather*}
hold.

For each $n\ge0$ the finite cocycle \prettyref{eq:6-1} gives 
\[
M_{n+1}\left(s_{i},s_{j};\iota_{k}\left(\omega\right)\right)=mM_{n}\left(\varphi_{k}\left(s_{i}\right),\varphi_{k}\left(s_{j}\right);\omega\right).
\]
Letting $n\to\infty$ along $\omega\in\Omega\setminus N\left(s_{1},\dots,s_{r}\right)$
yields \prettyref{eq:6-3}. The identity \prettyref{eq:6-4} follows
by applying \prettyref{eq:6-3} with $k=\omega_{1}$ and $\sigma\left(\omega\right)$
in place of $\omega$. 
\end{proof}

We use the diagonal values of the boundary kernel. For $s\in X_{2}$
define 
\begin{equation}
h\left(s;\omega\right):=M_{\infty}\left(s,s;\omega\right)\in\left[0,\infty\right)\qquad\left(\omega\in\Omega\right).\label{eq:h-definition}
\end{equation}
By \prettyref{thm:5-5} and \prettyref{eq:5-8} we have 
\begin{equation}
\mathbb{E}\left[h\left(s;\cdot\right)\right]=u_{\infty}\left(s\right)\qquad\left(s\in X_{2}\right).\label{eq:h-mean}
\end{equation}
The cocycle identity also gives the following relation for $h$..
\begin{cor}
\label{cor:h-cocycle}Fix $r\ge1$ and points $s_{1},\dots,s_{r}\in X_{2}$.
Then outside the same null set $N\left(s_{1},\dots,s_{r}\right)$
from \prettyref{thm:6-2}, one has 
\begin{equation}
h\left(s_{i};\omega\right)=mh\left(\varphi_{\omega_{1}}\left(s_{i}\right);\sigma\left(\omega\right)\right)\qquad\left(1\le i\le r\right).\label{eq:h-cocycle}
\end{equation}
\end{cor}

\begin{proof}
This is \prettyref{eq:6-4} with $s_{j}=s_{i}$. 
\end{proof}

We next show the pointwise domination of off-diagonal boundary values
by the diagonal boundary factors. This is the almost sure limit of
the $2\times2$ determinant bound.
\begin{cor}
\label{cor:6-4}Fix $s,t\in X_{2}$. There exists a $\mu$-null set
$N\left(s,t\right)\subset\Omega$ such that for every $\omega\in\Omega\setminus N\left(s,t\right)$,
\begin{equation}
\left|M_{\infty}\left(s,t;\omega\right)\right|^{2}\le h\left(s;\omega\right)h\left(t;\omega\right).\label{eq:6-8}
\end{equation}
\end{cor}

\begin{proof}
For each $n\ge0$ and each $\omega\in\Omega$, the $2\times2$ Gram
matrix 
\begin{multline*}
\left[\begin{matrix}M_{n}\left(s,s;\omega\right) & M_{n}\left(s,t;\omega\right)\\
M_{n}\left(t,s;\omega\right) & M_{n}\left(t,t;\omega\right)
\end{matrix}\right]\\
=m^{n}\left[\begin{matrix}K_{\infty}\left(\varphi_{\omega\mid n}\left(s\right),\varphi_{\omega\mid n}\left(s\right)\right) & K_{\infty}\left(\varphi_{\omega\mid n}\left(s\right),\varphi_{\omega\mid n}\left(t\right)\right)\\
K_{\infty}\left(\varphi_{\omega\mid n}\left(t\right),\varphi_{\omega\mid n}\left(s\right)\right) & K_{\infty}\left(\varphi_{\omega\mid n}\left(t\right),\varphi_{\omega\mid n}\left(t\right)\right)
\end{matrix}\right]
\end{multline*}
is positive semidefinite. Hence its determinant is nonnegative, which
gives 
\[
\left|M_{n}\left(s,t;\omega\right)\right|^{2}\le M_{n}\left(s,s;\omega\right)M_{n}\left(t,t;\omega\right)\qquad\left(n\ge0\right).
\]
By \prettyref{thm:5-5}, the limits $M_{n}\left(s,t;\omega\right)\to M_{\infty}\left(s,t;\omega\right)$,
$M_{n}\left(s,s;\omega\right)\to h\left(s;\omega\right)$, and $M_{n}\left(t,t;\omega\right)\to h\left(t;\omega\right)$
hold for $\mu$-almost every $\omega$. Passing to the limit along
such $\omega$ yields \prettyref{eq:6-8}. 
\end{proof}

We can normalize when the diagonal factors are positive. For $s,t\in X_{2}$
and $\omega\in\Omega$, set 
\begin{equation}
\widetilde{K}^{\omega}\left(s,t\right):=\begin{cases}
\dfrac{M_{\infty}\left(s,t;\omega\right)}{\left(h\left(s;\omega\right)h\left(t;\omega\right)\right)^{1/2}}, & h\left(s;\omega\right)h\left(t;\omega\right)>0,\\[1.2ex]
0, & h\left(s;\omega\right)h\left(t;\omega\right)=0.
\end{cases}\label{eq:6-9}
\end{equation}
By \prettyref{eq:6-8}, $|\widetilde{K}^{\omega}\left(s,t\right)|\le1$
when $h\left(s;\omega\right)h\left(t;\omega\right)>0$. Also, $\widetilde{K}^{\omega}\left(s,s\right)=1$
if $h\left(s;\omega\right)>0$.
\begin{cor}
\label{cor:6-5}Fix $r\ge1$ and points $s_{1},\dots,s_{r}\in X_{2}$.
There exists a $\mu$-null set $N\left(s_{1},\dots,s_{r}\right)\subset\Omega$
such that for every $\omega\in\Omega\setminus N\left(s_{1},\dots,s_{r}\right)$
the matrix 
\[
\left[M_{\infty}\left(s_{i},s_{j};\omega\right)\right]^{r}_{i,j=1}
\]
is positive semidefinite and satisfies the cocycle identity \prettyref{eq:6-4}.
Moreover, for every $\omega\in\Omega\setminus N\left(s_{1},\dots,s_{r}\right)$
the normalized matrix 
\[
\left[\widetilde{K}^{\omega}\left(s_{i},s_{j}\right)\right]^{r}_{i,j=1}
\]
is positive semidefinite after restricting to the indices $i$ for
which $h\left(s_{i};\omega\right)>0$.
\end{cor}

\begin{proof}
The first assertion is the combination of \prettyref{thm:5-6} and
\prettyref{thm:6-2}, after replacing the exceptional set by a union.
For the normalized matrix, fix $\omega$ outside this exceptional
set and let $I\subset\left\{ 1,\dots,r\right\} $ be the set of indices
with $h\left(s_{i};\omega\right)>0$. Write $D\left(\omega\right)$
for the diagonal matrix with diagonal entries $h\left(s_{i};\omega\right)^{1/2}$
for $i\in I$. Then by construction, 
\[
\left[M_{\infty}\left(s_{i},s_{j};\omega\right)\right]^{\left|I\right|}_{i,j\in I}=D\left(\omega\right)\left[\widetilde{K}^{\omega}\left(s_{i},s_{j}\right)\right]^{\left|I\right|}_{i,j\in I}D\left(\omega\right).
\]
Since $D\left(\omega\right)$ is invertible on the index set $I$,
the matrix $\left[\widetilde{K}^{\omega}\left(s_{i},s_{j}\right)\right]_{i,j\in I}$
is positive semidefinite whenever $\left[M_{\infty}\left(s_{i},s_{j};\omega\right)\right]_{i,j\in I}$
is positive semidefinite. 
\end{proof}

\section{Weighted boundary kernels and invariant majorants}\label{sec:7}

We now show that the boundary kernels constructed above generate an
explicit cone of positive definite kernels. The shift cocycle identifies
when such kernels are $L$-invariant. 

Let $f:\Omega\to\left[0,\infty\right]$ be a measurable function.
Define the domain 
\begin{equation}
X_{f}:=\left\{ s\in X_{2}:\int_{\Omega}f\left(\omega\right)h\left(s;\omega\right)d\mu\left(\omega\right)<\infty\right\} .\label{eq:7-1}
\end{equation}
For $s,t\in X_{f}$ we define a weighted boundary kernel by 
\begin{equation}
J_{f}\left(s,t\right):=\int_{\Omega}f\left(\omega\right)M_{\infty}\left(s,t;\omega\right)d\mu\left(\omega\right).\label{eq:7-2}
\end{equation}
The pointwise domination \prettyref{eq:6-8} ensures that \prettyref{eq:7-2}
is well-defined on $X_{f}\times X_{f}$.
\begin{lem}
\label{lem:7-1}Fix a measurable $f:\Omega\to\left[0,\infty\right]$.
If $s,t\in X_{f}$, then the integral \prettyref{eq:7-2} converges
absolutely and 
\begin{equation}
\left|J_{f}\left(s,t\right)\right|^{2}\le\left(\int_{\Omega}f\left(\omega\right)h\left(s;\omega\right)d\mu\left(\omega\right)\right)\left(\int_{\Omega}f\left(\omega\right)h\left(t;\omega\right)d\mu\left(\omega\right)\right).\label{eq:7-3}
\end{equation}
\end{lem}

\begin{proof}
Fix $s,t\in X_{f}$. By \prettyref{cor:6-4} we may choose a $\mu$-null
set outside of which $\left|M_{\infty}\left(s,t;\omega\right)\right|\le h\left(s;\omega\right)^{1/2}h\left(t;\omega\right)^{1/2}$.
Therefore 
\[
\int_{\Omega}f\left(\omega\right)\left|M_{\infty}\left(s,t;\omega\right)\right|d\mu\left(\omega\right)\le\int_{\Omega}f\left(\omega\right)h\left(s;\omega\right)^{1/2}h\left(t;\omega\right)^{1/2}d\mu\left(\omega\right).
\]
Apply Cauchy-Schwarz in $L^{2}\left(\Omega,\mu\right)$ to obtain
\begin{multline*}
\int_{\Omega}f\left(\omega\right)h\left(s;\omega\right)^{1/2}h\left(t;\omega\right)^{1/2}d\mu\left(\omega\right)\\
\le\left(\int_{\Omega}f\left(\omega\right)h\left(s;\omega\right)d\mu\left(\omega\right)\right)^{1/2}\left(\int_{\Omega}f\left(\omega\right)h\left(t;\omega\right)d\mu\left(\omega\right)\right)^{1/2}.
\end{multline*}
The right-hand side is finite because $s,t\in X_{f}$. This gives
absolute convergence of \prettyref{eq:7-2}. The inequality \prettyref{eq:7-3}
follows by applying the same Cauchy-Schwarz bound directly to the
integral in \prettyref{eq:7-2}. 
\end{proof}

We first check positivity of $J_{f}$. The boundary Gram matrices
from \prettyref{thm:5-6} are enough for this.
\begin{thm}
\label{thm:7-2}Fix a measurable $f:\Omega\to\left[0,\infty\right]$.
Then $J_{f}$ is a positive definite kernel on $X_{f}$. More precisely,
for every $r\ge1$ and every $s_{1},\dots,s_{r}\in X_{f}$, the Gram
matrix 
\[
\left[J_{f}\left(s_{i},s_{j}\right)\right]^{r}_{i,j=1}
\]
is positive semidefinite. 
\end{thm}

\begin{proof}
Fix $r\ge1$ and points $s_{1},\dots,s_{r}\in X_{f}$. By \prettyref{thm:5-6}
there exists a $\mu$-null set $N\left(s_{1},\dots,s_{r}\right)$
such that for every $\omega\in\Omega\setminus N\left(s_{1},\dots,s_{r}\right)$
the matrix 
\[
G\left(\omega\right):=\left[M_{\infty}\left(s_{i},s_{j};\omega\right)\right]^{r}_{i,j=1}
\]
is positive semidefinite.

Define the matrix-valued function $H\left(\omega\right):=f\left(\omega\right)G\left(\omega\right)$.
Since $f\ge0$, each $H\left(\omega\right)$ is positive semidefinite
for $\omega\in\Omega\setminus N\left(s_{1},\dots,s_{r}\right)$.

We claim that each entry of $H$ is integrable. Fix $i,j$. By \prettyref{eq:6-8},
outside a $\mu$-null set we have 
\[
\left|f\left(\omega\right)M_{\infty}\left(s_{i},s_{j};\omega\right)\right|\le f\left(\omega\right)h\left(s_{i};\omega\right)^{1/2}h\left(s_{j};\omega\right)^{1/2}.
\]
Since $s_{i},s_{j}\in X_{f}$, Cauchy-Schwarz gives 
\begin{multline*}
\int_{\Omega}f\left(\omega\right)h\left(s_{i};\omega\right)^{1/2}h\left(s_{j};\omega\right)^{1/2}d\mu\left(\omega\right)\\
\le\left(\int_{\Omega}f\left(\omega\right)h\left(s_{i};\omega\right)d\mu\left(\omega\right)\right)^{1/2}\left(\int_{\Omega}f\left(\omega\right)h\left(s_{j};\omega\right)d\mu\left(\omega\right)\right)^{1/2}<\infty.
\end{multline*}
Thus $H_{i,j}$ is integrable and 
\[
\int_{\Omega}H_{i,j}\left(\omega\right)\,d\mu\left(\omega\right)=\int_{\Omega}f\left(\omega\right)M_{\infty}\left(s_{i},s_{j};\omega\right)\,d\mu\left(\omega\right)=J_{f}\left(s_{i},s_{j}\right).
\]
To prove that the Gram matrix of $J_{f}$ is positive semidefinite,
fix $c\in\mathbb{C}^{r}$. For $\omega\in\Omega\setminus N\left(s_{1},\dots,s_{r}\right)$,
positivity of $H\left(\omega\right)$ gives $c^{*}H\left(\omega\right)c\ge0$.
Moreover, $c^{*}H\left(\omega\right)c$ is integrable because it is
a finite sum of integrable entries. Therefore 
\begin{align*}
c^{*}\left[J_{f}\left(s_{i},s_{j}\right)\right]^{r}_{i,j=1}c & =\sum^{r}_{i,j=1}\overline{c_{i}}c_{j}J_{f}\left(s_{i},s_{j}\right)\\
 & =\int_{\Omega}\sum^{r}_{i,j=1}\overline{c_{i}}c_{j}f\left(\omega\right)M_{\infty}\left(s_{i},s_{j};\omega\right)d\mu\left(\omega\right)\\
 & =\int_{\Omega}c^{*}H\left(\omega\right)cd\mu\left(\omega\right)\ge0.
\end{align*}
Since $c$ was arbitrary, the Gram matrix is positive semidefinite. 
\end{proof}

We now apply $L$ to $J_{f}$. The result is simple. The weight is
shifted.
\begin{thm}
\label{thm:7-3}Fix a measurable $f:\Omega\to\left[0,\infty\right]$.
Define $f\circ\sigma$ by $f\circ\sigma\left(\omega\right)=f\left(\sigma\left(\omega\right)\right)$.
Then for all $s,t\in X_{f\circ\sigma}$, 
\begin{equation}
\left(LJ_{f}\right)\left(s,t\right)=J_{f\circ\sigma}\left(s,t\right).\label{eq:7-4}
\end{equation}
In particular, if $f=f\circ\sigma$ $\mu$-almost surely, then $J_{f}$
is $L$-invariant on $X_{f}$. 
\end{thm}

\begin{proof}
Fix $s,t\in X_{f\circ\sigma}$. We first note that $\varphi_{i}\left(s\right),\varphi_{i}\left(t\right)\in X_{f}$
for every $i\in\left\{ 1,\dots,m\right\} $. Indeed, using the diagonal
cocycle \prettyref{eq:6-3} on the diagonal, for $\mu$-almost every
$\omega$, 
\[
h\left(s;\iota_{i}\left(\omega\right)\right)=mh\left(\varphi_{i}\left(s\right);\omega\right),\quad h\left(t;\iota_{i}\left(\omega\right)\right)=mh\left(\varphi_{i}\left(t\right);\omega\right).
\]
Hence 
\begin{align*}
\int_{\Omega}f\left(\omega\right)h\left(\varphi_{i}\left(s\right);\omega\right)d\mu\left(\omega\right) & =\frac{1}{m}\int_{\Omega}f\left(\omega\right)h\left(s;\iota_{i}\left(\omega\right)\right)d\mu\left(\omega\right)\\
 & =\int_{\left\{ \eta\in\Omega:\eta_{1}=i\right\} }f\left(\sigma\left(\eta\right)\right)h\left(s;\eta\right)d\mu\left(\eta\right)\\
 & \le\int_{\Omega}\left(f\circ\sigma\right)\left(\eta\right)h\left(s;\eta\right)d\mu\left(\eta\right),
\end{align*}
which is finite because $s\in X_{f\circ\sigma}$. Thus $\varphi_{i}\left(s\right)\in X_{f}$,
and similarly $\varphi_{i}\left(t\right)\in X_{f}$.

Now $J_{f}\left(\varphi_{i}\left(s\right),\varphi_{i}\left(t\right)\right)$
is well-defined for each $i$, and we compute 
\[
\left(LJ_{f}\right)\left(s,t\right)=\sum^{m}_{i=1}J_{f}\left(\varphi_{i}\left(s\right),\varphi_{i}\left(t\right)\right)=\sum^{m}_{i=1}\int_{\Omega}f\left(\omega\right)M_{\infty}\left(\varphi_{i}\left(s\right),\varphi_{i}\left(t\right);\omega\right)d\mu\left(\omega\right).
\]
Since the sum is finite and each integral converges absolutely, we
may interchange sum and integral: 
\[
\left(LJ_{f}\right)\left(s,t\right)=\int_{\Omega}f\left(\omega\right)\sum^{m}_{i=1}M_{\infty}\left(\varphi_{i}\left(s\right),\varphi_{i}\left(t\right);\omega\right)d\mu\left(\omega\right).
\]
We compute the inner sum using the prefix cocycle \prettyref{eq:6-3}:
for $\mu$-almost every $\omega$ and every $i$, 
\[
M_{\infty}\left(\varphi_{i}\left(s\right),\varphi_{i}\left(t\right);\omega\right)=\frac{1}{m}M_{\infty}\left(s,t;\iota_{i}\left(\omega\right)\right).
\]
Hence 
\[
\sum^{m}_{i=1}M_{\infty}\left(\varphi_{i}\left(s\right),\varphi_{i}\left(t\right);\omega\right)=\frac{1}{m}\sum^{m}_{i=1}M_{\infty}\left(s,t;\iota_{i}\left(\omega\right)\right),
\]
and therefore 
\[
\left(LJ_{f}\right)\left(s,t\right)=\int_{\Omega}f\left(\omega\right)\frac{1}{m}\sum^{m}_{i=1}M_{\infty}\left(s,t;\iota_{i}\left(\omega\right)\right)d\mu\left(\omega\right).
\]
Finally, use the product structure of $\mu$. The map 
\[
T:\left\{ 1,\dots,m\right\} \times\Omega\to\Omega,\qquad T\left(i,\omega\right)=\iota_{i}\left(\omega\right),
\]
pushes forward the product measure $\left(\frac{1}{m}\sum^{m}_{i=1}\delta_{i}\right)\otimes\mu$
to $\mu$, and $\sigma\circ T\left(i,\omega\right)=\omega$. Therefore,
by a change of variables, 
\begin{align*}
\left(LJ_{f}\right)\left(s,t\right) & =\int_{\Omega}\frac{1}{m}\sum^{m}_{i=1}f\left(\omega\right)M_{\infty}\left(s,t;\iota_{i}\left(\omega\right)\right)d\mu\left(\omega\right)\\
 & =\int_{\left\{ 1,\dots,m\right\} \times\Omega}f\left(\omega\right)M_{\infty}\left(s,t;T\left(i,\omega\right)\right)d\left(\left(\frac{1}{m}\sum^{m}_{i=1}\delta_{i}\right)\otimes\mu\right)\left(i,\omega\right)\\
 & =\int_{\Omega}f\left(\sigma\left(\eta\right)\right)M_{\infty}\left(s,t;\eta\right)d\mu\left(\eta\right)\\
 & =\int_{\Omega}\left(f\circ\sigma\right)\left(\eta\right)M_{\infty}\left(s,t;\eta\right)d\mu\left(\eta\right)\\
 & =J_{f\circ\sigma}\left(s,t\right),
\end{align*}
which is \prettyref{eq:7-4}.

If $f=f\circ\sigma$ $\mu$-almost surely, then $X_{f}=X_{f\circ\sigma}$
and $J_{f}=J_{f\circ\sigma}$ on $X_{f}$, hence $\left(LJ_{f}\right)\left(s,t\right)=J_{f}\left(s,t\right)$
for all $s,t\in X_{f}$. 
\end{proof}

\begin{cor}
\label{cor:7-4}Let $f\equiv1$ on $\Omega$. Then $X_{1}=X_{2}$
and 
\begin{equation}
J_{1}\left(s,t\right)=K_{\infty}\left(s,t\right)\qquad\left(s,t\in X_{2}\right).\label{eq:7-6}
\end{equation}
In particular, $J_{1}$ is $L$-invariant on $X_{2}$.
\end{cor}

\begin{proof}
If $f\equiv1$, then by \prettyref{eq:h-mean} we have 
\[
\int_{\Omega}h\left(s;\omega\right)d\mu\left(\omega\right)=u_{\infty}\left(s\right)<\infty\qquad\left(s\in X_{2}\right),
\]
so $X_{2}\subset X_{1}$. Conversely, $X_{1}\subset X_{2}$ is immediate
from the definition since $X_{f}\subset X_{2}$ for every $f$. For
$s,t\in X_{2}$, the identity \prettyref{eq:7-6} is \prettyref{eq:5-8}.
The final statement follows from \prettyref{thm:7-3} and the fact
that $1\circ\sigma=1$. 
\end{proof}

We include the basic ordering relation with respect to the weight.
\begin{prop}
\label{prop:7-5} Fix measurable $f,g:\Omega\to\left[0,\infty\right]$
with $f\le g$ $\mu$-almost surely. Then $X_{g}\subset X_{f}$ and
\begin{equation}
J_{f}\le J_{g}\label{eq:7-5}
\end{equation}
in the Loewner order on $X_{g}$. Equivalently, for every $r\ge1$
and every $s_{1},\dots,s_{r}\in X_{g}$, 
\[
\left[J_{f}\left(s_{i},s_{j}\right)\right]^{r}_{i,j=1}\le\left[J_{g}\left(s_{i},s_{j}\right)\right]^{r}_{i,j=1}
\]
in the positive semidefinite cone. In particular, 
\[
J_{f}\left(s,s\right)\le J_{g}\left(s,s\right)\qquad\left(s\in X_{g}\right).
\]
\end{prop}

\begin{proof}
If $f\le g$ and $\int_{\Omega}g\left(\omega\right)h\left(s;\omega\right)d\mu\left(\omega\right)<\infty$,
then $\int_{\Omega}f\left(\omega\right)h\left(s;\omega\right)d\mu\left(\omega\right)<\infty$,
hence $X_{g}\subset X_{f}$.

Fix $r\ge1$ and $s_{1},\dots,s_{r}\in X_{g}$. By \prettyref{thm:5-6}
there exists a $\mu$-null set $N\left(s_{1},\dots,s_{r}\right)$
such that for every $\omega\in\Omega\setminus N\left(s_{1},\dots,s_{r}\right)$
the matrix 
\[
G\left(\omega\right):=\left[M_{\infty}\left(s_{i},s_{j};\omega\right)\right]^{r}_{i,j=1}
\]
is positive semidefinite. Since $g-f\ge0$ $\mu$-almost surely, the
matrix 
\[
\left(g\left(\omega\right)-f\left(\omega\right)\right)G\left(\omega\right)
\]
is positive semidefinite for $\mu$-almost every $\omega$.

We claim that each entry of $\left(g-f\right)\left(\omega\right)G\left(\omega\right)$
is integrable. Fix $i,j$. By \prettyref{eq:6-8}, outside a $\mu$-null
set, 
\begin{align*}
\left|\left(g\left(\omega\right)-f\left(\omega\right)\right)M_{\infty}\left(s_{i},s_{j};\omega\right)\right| & \le g\left(\omega\right)\left|M_{\infty}\left(s_{i},s_{j};\omega\right)\right|\\
 & \le g\left(\omega\right)h\left(s_{i};\omega\right)^{1/2}h\left(s_{j};\omega\right)^{1/2}.
\end{align*}
Since $s_{i},s_{j}\in X_{g}$, Cauchy-Schwarz gives 
\begin{multline*}
\int_{\Omega}g\left(\omega\right)h\left(s_{i};\omega\right)^{1/2}h\left(s_{j};\omega\right)^{1/2}d\mu\left(\omega\right)\\
\le\left(\int_{\Omega}g\left(\omega\right)h\left(s_{i};\omega\right)d\mu\left(\omega\right)\right)^{1/2}\left(\int_{\Omega}g\left(\omega\right)h\left(s_{j};\omega\right)d\mu\left(\omega\right)\right)^{1/2}<\infty.
\end{multline*}
Thus each entry is integrable, and integrating entrywise yields 
\[
\left[J_{g}\left(s_{i},s_{j}\right)-J_{f}\left(s_{i},s_{j}\right)\right]^{r}_{i,j=1}=\int_{\Omega}\left(g\left(\omega\right)-f\left(\omega\right)\right)G\left(\omega\right)d\mu\left(\omega\right),
\]
which is positive semidefinite. This is  the kernel order inequality
\prettyref{eq:7-5}. The diagonal inequality is the special case $r=1$. 
\end{proof}

Thus $f\mapsto J_{f}$ gives positive definite kernels on the corresponding
sets $X_{f}$. By \prettyref{thm:7-3}, shift-invariance of $f$ gives
$L$-invariance of $J_{f}$. There is also the order from \prettyref{prop:7-5}.
If the weights are ordered, so are the kernels.

For a general $f$, consider all its shifts at once. Let 
\[
f^{*}\left(\omega\right):=\sup_{n\ge0}f\left(\sigma^{n}\left(\omega\right)\right).
\]
This gives an upper bound for the kernels $L^{n}J_{f}$. It also has
$f^{*}\circ\sigma\le f^{*}$, which gives an $L$-superharmonic kernel.
\begin{thm}
\label{thm:7-6} Fix a measurable function $f:\Omega\to\left[0,\infty\right]$,
and define 
\[
f^{*}\left(\omega\right):=\sup_{n\ge0}f\left(\sigma^{n}\left(\omega\right)\right)\qquad\left(\omega\in\Omega\right).
\]
Then $f^{*}:\Omega\to\left[0,\infty\right]$ is measurable. Moreover,
for every $n\ge0$ one has 
\[
X_{f^{*}}\subset X_{f\circ\sigma^{n}}
\]
and the kernel inequality 
\begin{equation}
\left(L^{n}J_{f}\right)\le J_{f^{*}}\label{eq:7-7}
\end{equation}
holds in the Loewner order on $X_{f^{*}}$. In addition, $J_{f^{*}}$
is $L$-superharmonic in the sense that 
\begin{equation}
\left(LJ_{f^{*}}\right)\le J_{f^{*}}\label{eq:7-8}
\end{equation}
in the Loewner order on $X_{f^{*}}$. 
\end{thm}

\begin{proof}
Measurability of $f^{*}$ follows because $f^{*}$ is the pointwise
supremum of the countable family of measurable functions $\omega\mapsto f\left(\sigma^{n}\left(\omega\right)\right)$.

Fix $n\ge0$. Since $f\circ\sigma^{n}\le f^{*}$ pointwise, \prettyref{prop:7-5}
applied with $f\circ\sigma^{n}$ in place of $f$ and $f^{*}$ in
place of $g$ gives $X_{f^{*}}\subset X_{f\circ\sigma^{n}}$ and 
\begin{equation}
J_{f\circ\sigma^{n}}\le J_{f^{*}}\;\text{on }X_{f^{*}}.\label{eq:7-9}
\end{equation}

Next we identify $L^{n}J_{f}$ with $J_{f\circ\sigma^{n}}$ on $X_{f\circ\sigma^{n}}$.
For $n=0$ this is trivial. For $n\ge1$, iterate \prettyref{thm:7-3}:
for every $k\ge0$ and all $s,t\in X_{f\circ\sigma^{k+1}}$, 
\[
\left(LJ_{f\circ\sigma^{k}}\right)\left(s,t\right)=J_{\left(f\circ\sigma^{k}\right)\circ\sigma}\left(s,t\right)=J_{f\circ\sigma^{k+1}}\left(s,t\right).
\]
By induction on $k$ this yields 
\begin{equation}
\left(L^{n}J_{f}\right)\left(s,t\right)=J_{f\circ\sigma^{n}}\left(s,t\right)\qquad\left(s,t\in X_{f\circ\sigma^{n}}\right).\label{eq:7-10}
\end{equation}
In particular, since $X_{f^{*}}\subset X_{f\circ\sigma^{n}}$, the
identity \prettyref{eq:7-10} holds for all $s,t\in X_{f^{*}}$.

Combining \prettyref{eq:7-10} with \prettyref{eq:7-9} gives the
claimed majorization \prettyref{eq:7-7} on $X_{f^{*}}$.

Finally, we prove \prettyref{eq:7-8}. By definition, 
\[
\left(f^{*}\circ\sigma\right)\left(\omega\right)=\sup_{n\ge0}f\left(\sigma^{n+1}\left(\omega\right)\right)\le\sup_{n\ge0}f\left(\sigma^{n}\left(\omega\right)\right)=f^{*}\left(\omega\right),
\]
so $f^{*}\circ\sigma\le f^{*}$ pointwise. \prettyref{prop:7-5} then
gives $X_{f^{*}}\subset X_{f^{*}\circ\sigma}$ and 
\begin{equation}
J_{f^{*}\circ\sigma}\le J_{f^{*}}\;\text{on }X_{f^{*}}.\label{eq:7-11}
\end{equation}
Since $X_{f^{*}}\subset X_{f^{*}\circ\sigma}$, \prettyref{thm:7-3}
applied to the weight $f^{*}$ yields 
\[
\left(LJ_{f^{*}}\right)\left(s,t\right)=J_{f^{*}\circ\sigma}\left(s,t\right)\qquad\left(s,t\in X_{f^{*}}\right).
\]
Together with \prettyref{eq:7-11}, this gives $\left(LJ_{f^{*}}\right)\le J_{f^{*}}$
on $X_{f^{*}}$, which is \prettyref{eq:7-8}. 
\end{proof}

\section*{Declarations}

\subsection*{Funding}

No funding was received for conducting this study.

\subsection*{Competing interests}

The author has no competing interests to declare that are relevant
to the content of this article.

\subsection*{Data availability}

No datasets were generated or analyzed for this study. All results
are established analytically, and the finite-dimensional examples
are fully specified within the article.

\bibliographystyle{amsalpha}
\bibliography{ref}

\end{document}